\documentclass[12pt,oneside,reqno]{amsart}
\usepackage[latin1]{inputenc}
\usepackage[english]{babel}
\usepackage{mathrsfs}
\usepackage{indentfirst}
\usepackage{paralist}
\usepackage{amssymb}
\usepackage{amsthm}
\usepackage{amsmath}
\usepackage{amscd}
\usepackage{graphicx}
\usepackage[colorlinks=true]{hyperref}
\usepackage[margin=2.5cm]{geometry}
\usepackage{cite}
\usepackage{lineno} %% adds numbers to the lines
\usepackage{mathptmx}

\usepackage{xcolor}
\definecolor{ForestGreen}{rgb}{0.13, 0.55, 0.13}

\theoremstyle{plain} 
\newtheorem{theorem}{Theorem}[section] 
\newtheorem*{theorem*}{Main Result}
\newtheorem{thm*}{Known result}
\newtheorem{metthm}{Main result}
 
\newtheorem{lemma}[theorem]{Lemma}
\newtheorem{prop}[theorem]{Proposition}
\newtheorem{defn}[theorem]{Definition}

\theoremstyle{definition}

\theoremstyle{remark}
\newtheorem{remark}[theorem]{Remark}

\numberwithin{equation}{section} 

\def\Xint#1{\mathchoice
	{\XXint\displaystyle\textstyle{#1}}%
	{\XXint\textstyle\scriptstyle{#1}}%
	{\XXint\scriptstyle\scriptscriptstyle{#1}}%
	{\XXint\scriptscriptstyle\scriptscriptstyle{#1}}%
	\!\int}
\def\XXint#1#2#3{{\setbox0=\hbox{$#1{#2#3}{\int}$}
		\vcenter{\hbox{$#2#3$}}\kern-.5\wd0}}

\def\dashint{\Xint-}

\newcommand{\eqlab}[1]{\begin{equation}  \begin{aligned}#1 \end{aligned}\end{equation}} %this substitutes \begin{equation}\begin{aligned} ... 
\newcommand{\bgs}[1]{\begin{equation*} \begin{aligned}#1\end{aligned}\end{equation*}} %this substitutes \begin{equation*}\begin{aligned} ...
\newcommand{\sys}[2][]{\begin{equation*}#1  \left\{\begin{aligned}#2\end{aligned}\right.\end{equation*}}%this substitutes \begin{equation*}\begin{aligned} ... with the graph parenthesis 

\newcommand{\R}{\ensuremath{\mathbb{R}}}
\newcommand{\Rn}{\ensuremath{\mathbb{R}^n}}
\newcommand{\N}{\ensuremath{\mathbb{N}}}
\newcommand{\eps}{\ensuremath{\varepsilon}}

\newcommand{\Sg}{\mathcal{S}g}
\newcommand{\Ll}{\mathcal L}
\newcommand{\Co}{\mathcal C}

\newcommand{\Fc}{\mathcal F}
\newcommand{\Ha}{\mathcal H}

\newcommand{\G}{\mathcal G}
\newcommand{\A}{\mathcal A}
\newcommand{\Nl}{\mathcal N}
\newcommand{\Op}{\mathcal O}

\newcommand{\lra}{\longrightarrow}
\newcommand{\B}{\mathfrak B}
\newcommand{\ball}{\mathcal B} %ball in $R^{n+1}$
\newcommand{\W}{\mathcal W}

\DeclareMathOperator{\diam}{diam}
\DeclareMathOperator{\Per}{Per}
\DeclareMathOperator{\Tail}{Tail}
\DeclareMathOperator{\Tails}{Tail_s}

\newcommand{\F}{\mathcal F}
\newcommand{\h}{\mathscr{H}}

\newcommand{\kers}{|x-y|^{n-1+s}}

\newcommand{\dkers}{\frac{dx\,dy}{|x-y|^{n-1+s}}}

\DeclareMathOperator{\dist}{dist}

\makeatletter
\def\l@subsection{\@tocline{2}{0pt}{2.5pc}{5pc}{}}
\makeatother

\renewcommand{\le}{\leqslant}
\renewcommand{\leq}{\leqslant}
\renewcommand{\ge}{\geqslant}
\renewcommand{\geq}{\geqslant}

\begin{document}
	\author{C. Bucur, L. Lombardini}
	\title[Asymptotics as $s\searrow 0$ of the nonlocal nonparametric Plateau problem]{Asymptotics as $s\searrow 0$ of the nonlocal nonparametric Plateau problem with obstacles
	}
	\begin{abstract}
	In this paper, we introduce a functional and a geometric setting for an obstacle problem for nonlocal minimal graphs. In particular we study existence of solutions, a priori estimates, and we prove the equivalence of the two settings. We then observe a striking stickiness phenomena when the fractional parameter is small and the data at infinity is not too large: the nonlocal minimal graphs adhere entirely to the obstacle and leave the remainder of the domain asymptotically empty. We thus provide a class of examples where continuity of nonlocal minimal graphs across the boundary and across the obstacle may fail.
	\end{abstract}
	
	\maketitle
	\setcounter{tocdepth}{1}
	%\tableofcontents

	%%%%%SECTION%%%% 
	\section{Introduction}
	
	Nonlocal minimal surfaces were first introduced in \cite{nms} and in recent decades much effort has been put into analyzing their analytical and geometric properties,  as well as comparing them to their classical local counterparts, namely minimal surfaces of co-dimension one in the euclidean space, via  a De Giorgi approach, see \cite{DG5}. One behavior of nonlocal minimal surfaces that stood out, first noticed in \cite{boundary} in the two dimensional case, was named  \emph{stickiness to the boundary} and refers to the fact that
	nonlocal minimal surfaces have a tendency to attach to the boundary of the domain in which the minimization problem is set. Following a lead from \cite{boundary}, a general example of \emph{complete stickiness} in a bounded domain was studied in \cite{BucurLombardini} when the fractional parameter $s$ is small enough. 
	
	Nonparametric nonlocal minimal surfaces or, otherwise said, \emph{nonlocal minimal graphs}, are nonlocal minimal surfaces which are graphs in some direction. 
	
	In the  present paper we extend the observations in \cite{BucurLombardini} to the case of nonlocal minimal graphs, extending the complete stickiness behavior to minimizing sets in unbounded cylinders.  The stickiness phenomena are also important in view of regularity results, pointing out that continuity across the boundary may fail. Stickiness seems to be typical for nonlocal minimal sets as observed in \cite{DSVnmsstick20,DSVbdry20,trace}. Precisely, it is showed in \cite{DSVnmsstick20}  that either nonlocal minimal graphs are continuous across the boundary, or a small perturbation of the exterior data can produce stickiness. We also point out that in the recent paper \cite{trace} the regularity of the trace is studied: when $\Omega$ is smooth, locally around points of stickiness, i.e. where continuity across the boundary fails, the trace is a $C^{1,\gamma}$ (for some $\gamma>0$) $(n-1)$-dimensional surface in $\R^{n+1}$.
	
	The main goal of the paper is to give a complete stickiness result for nonlocal minimal graphs for small values of the fractional parameter. In other terms, it provides an example of a minimal graph which fails to be continuous across the boundary. To achieve this goal, we introduce the suitable functional and equivalent geometric framework for a nonparametric Plateau problem with obstacles. 
	
	\smallskip

	We recall that in \cite{nms} a \emph{nonlocal minimal surface} is defined as the boundary of a \emph{nonlocal minimal set}, a set that minimizes the fractional perimeter functional.  Given $\mathcal O\subset \R^{n+1}$ a bounded open set and $E_0\subset \R^{n+1}$,
	the set $E\subset \R^{n+1}$ such that $E\cap \Co \mathcal O=E_0$ is an $s$-minimal set  in $\mathcal O$ with respect to $E_0$ if $\Per_s(E, \mathcal O)$ is finite and if
	\eqlab{ \label{minper} \Per_s(E, \mathcal O) \leq \Per_s(F, \mathcal O) \qquad \mbox{ for all } F\subset \R^{n+1} \mbox{ such that } \, \,  F\cap \Co \mathcal O=E_0\cap \Co \mathcal O,}
	with $\Per_s$ defined in \eqref{perimeter}. 
	We point out that in our setting the definition of an $s$-minimal set needs some adaptation, since the domain in which we minimize the perimeter is unbounded (the infinite cylinder), so we must work with local minimizers (see, for reference, \cite[Section 4.4]{LucaApprox}). 
	
	The study of minimal sets in the class of subgraphs is justified by \cite[Theorem 1.1]{graph} which states that nonlocal minimal surfaces that have a continuous subgraph as fixed exterior data are continuous subgraphs inside the smooth domain of reference. Furthermore, \cite[Theorem 1.9]{CozziLombardini} gives yet another motivation for working in the class of subgraphs: if the exterior data $E_0\subset \Co \mathcal O$ in \eqref{minper} is a subgraph then the perimeter  of a set decreases if such set is replaced by a subgraph (built with a ``vertical rearrangement '' of the set itself). 
	
	We consider thus the following problem. Given $\Omega\subset \Rn$ a bounded open set with smooth  boundary, an open set $A\subset \Omega$, the exterior data given by a function $\varphi \colon \Rn \to \R$ and the obstacle given by $\psi \colon A \to \R$, we want to find a function $u\colon \Rn \to \R$ such that 
	\eqlab{ \label{mainpb1}
		\mbox{ (i) }&\; u(x)=\varphi(x) \quad \mbox{ in }  \Co \Omega
		\\
		\mbox{ (ii) } &\; u(x) \geq \psi(x) \quad  \mbox{ in } A\subset \Omega,
		\\
		\mbox{ (iii) }&\; \mbox{ the subgraph of $u$ is an $s$ minimal set for the fractional perimeter in $\Omega\times \R$, among all  }\\ &\; \mbox{subgraphs of functions verifying (i) and (ii) }.} 
	We adopt this \emph{geometric point of view} in Subsection \ref{setting}.

	We mention that the geometric nonlocal obstacle problem, in which both the obstacle and the exterior data are sets and the domain of minimization is bounded has been considered in \cite{CDSS16}. 
	
	\medskip 
	
	The problem \eqref{mainpb1} can be defined also from a purely functional perspective. Indeed, in \cite{CozziLombardini}, a functional setting for  nonlocal minimal graphs was introduced. In Subsection \ref{functsett}, we adopt this functional approach in presence of obstacles, dealing with the problem \eqref{mainpb1} in which we  rewrite  (iii). We look for a function $u\colon \Rn \to \R$ such that
	\eqlab{ \label{mainpb2}
		\mbox{ (i) }&\; u(x)=\varphi(x) \quad \mbox{ in }  \Co \Omega
		\\
		\mbox{ (ii) } &\; u(x) \geq \psi(x) \quad  \mbox{ in } A\subset \Omega,
		\\
		\mbox{ (iii') }&\; \F_s(u, \Omega) \leq \F_s(v, \Omega) \, \, \mbox{ for all competitors } v ,}
	with $\F_s$ subsequently defined in \eqref{area}. 
	Notice that an admissible competitor satisfies (i) and (ii). 
	
	The local version of this problem has been studied in several papers, we mention for instance \cite{Giusti}. %dobbiamo citare altri? 
	
	Concerning our nonlocal obstacle problem, we first obtain some preliminary results on problems \eqref{mainpb1}, \eqref{mainpb2}. We give here a short list. 
	\begin{itemize}
		\item Under reasonable assumptions on $\varphi$ and $\psi$, problem \eqref{mainpb2}  admits a unique solution $u_s$, see Theorems \ref{Dirichlet_obstacle}  and \ref{Dirichlet_obstacle1}. 
		\item Some a-priori $L^\infty(\Omega)$ bounds on the solution as well as $L^\infty_{\textup{loc}}(\Omega)$ and $W^{s,1}(\Omega)$  hold (see, respectively, Theorems \ref{Dirichlet_obstacle}, \ref{Dirichlet_obstacle1}  and \ref{inftyloc}). 
		\item If $u$ solves \eqref{mainpb2}, then it solves also the geometric problem \eqref{mainpb1}, and even more generally the subgraph of $u$ turns out to be a minimizer among all sets and not just among subgraphs, see Theorems \ref{equiv}, \ref{equiv1}.
		\item An Euler-Lagrange equation can be put into evidence, see  Theorem \ref{eulero}.
	\end{itemize}
	
	Once we establish this setting for the obstacle problem, it is interesting to study the asymptotics as $s\to 0$ of the sequence of solutions $u_s$.
	
	\medskip
	
	Concerning these asymptotics, the first paper that analyzed the behavior of the fractional perimeter for small values of $s$ is \cite{asympt1}. As $s$ becomes smaller, that data far away from $\Omega$ becomes predominant. In order to mathematically encode such a behavior "at infinity" of a set $E\subset \R^{n+1}$ for small values of $s$, in~ \cite[(2.2)]{asympt1} the authors define the function set $\alpha$ in the following way
	\bgs{\label{alpha} \alpha(E)=\lim_{s\to 0^+} s\int_{\Co \ball_1} \frac{\chi_E(Y)}{|Y|^{n+1+s}}\, dY.  } 
	Here  $\ball_1$ is the ball in $\R^{n+1}$ of radius one, centered at the origin.\\
	We notice that the set function $\alpha$ may not exist even for smooth sets $E$ (Examples 2.8 and 2.9 in \cite{asympt1}). For this reason, we use the set functions introduced in \cite{BucurLombardini},
	\eqlab{\label {baralpha1} \overline \alpha (E):= \limsup_{s\to 0^+} s\int_{\Co \ball_1} \frac{\chi_E(Y)}{|Y|^{n+1+s}}\, dY ,}
	and $ \underline \alpha(E) $ as the corresponding $\liminf$.
	%:= \liminf_{s\to 0^+} s\int_{\Co \ball_1} \frac{\chi_E(Y)}{|Y|^{n+1+s}}\, dY,}  
	The limit as $s\to 0$ of the renormalized fractional perimeter of a set, see  \cite[Theorem 2.5]{asympt1},  is equal to a linear combination of the weighted measure at infinity (the function set $\alpha$) and of the Lebesgue measure of $E\cap\mathcal O$, i.e.
	\bgs{  \lim_{s\to 0^+} s\Per_s(E,\mathcal O)= &\; \alpha( E) |\mathcal O| +(\omega_{n+1}-2 \alpha(E)) | E \cap \mathcal O| 	,}
	whenever $\mathcal O \subset \R^{n+1} $ is  a bounded open set with smooth boundary, $E \subset \R^{n+1}$ has  finite $s_0$-perimeter, for some $s_0\in (0,1)$ and $\alpha(E)$ exists. 
	One could guess from here that if 
	\[ 2\alpha(E)< \omega_{n+1}\]
	then minimizing the $s$-perimeter would mean   to minimize the Lebesgue measure of $E\cap \mathcal O$, which would give $E\cap \mathcal O= \emptyset$. As a matter of fact, this intuition can be rigorously pinned down. 
	The result in \cite[Theorem 1.7]{BucurLombardini} says that if $\mathcal O$ is a bounded open set with $C^2$ boundary, the exterior data
	$E_0:=E\cap \Co \mathcal O$ is such that  \eqlab{ \label{alphacond} \overline \alpha(E_0)<\frac{\omega_{n+1}}{2},} and it does not completely surround $\mathcal O$, i.e. there exists $R>0$ and $X_0\in \partial \mathcal O$ such that 		
	%\eqlab{\label{condplus}
	$\ball_R(X_0)\setminus \Omega \subset \Co E_0,$
	%}
	then there is some $\bar s\in (0,1)$ small enough such that for all $s<\bar s$, if $E$ is an $s$-minimal set in $\mathcal O$ with exterior data $E_0$, then 
	\[  E\cap \mathcal O=\emptyset .\]
	If we allow the possibility that the exterior data $E_0$ completely surrounds $\mathcal O$, then the minimal sets become either empty, or arbitrarily ``dense'' (in a topological sense) inside $\mathcal O$. This more general result is stated in \cite[Theorem 1.4]{BucurLombardini}.
	
	Our main result in this regard, contained in Section \ref{two}, shows a stickiness behavior of solutions to our problem \eqref{mainpb1}/\eqref{mainpb2}. 
	\begin{metthm} Let $\varphi \colon \Rn \to \R$ such that
		\[
		\overline{\alpha}\big(\{(x,x_{n+1})\,|\,x_{n+1}<\varphi(x)\}\big)<\frac{\omega_{n+1}}{2},
		\]
		let $\psi\in C(\overline{A})$ be regular enough and let $\eps\in[0,1]$. For every $s\in(0,1)$, let $u_s$ be the solution of the obstacle problem  \eqref{mainpb2}, with respect to the obstacle $\eps\psi$. Then:
		\begin{itemize}
			\item if $A=\Omega$, there exists $s_0\in(0,1)$, not depending on $\eps$, such that
			\[
			u_s=\eps\psi\quad\mbox{a.e. in }\Omega,
			\]
			for every $s\in(0,s_0)$;
			\item if $A\subset\subset\Omega$, for every $k\geq 0$ big enough, there exists $s_k\in(0,1)$, not depending on $\eps$, such that
			\bgs{
				u_s\le-k\quad\mbox{ a.e. in }\Omega\setminus A\quad\mbox{ and }
				\quad u_s=\eps\psi\quad\mbox{ a.e. in }A,
			}
			for every $s\in(0,s_k)$.
			In particular
			$$\lim_{s\to0}u_s(x)=-\infty,\quad\mbox{ uniformly in }x\in\Omega\setminus A.$$
		\end{itemize}
	\end{metthm}

	To be precise, taking $\Omega$  a bounded and connected open set with smooth boundary,  $A\subset\subset \Omega$ open with smooth boundary, $\varphi$ locally bounded with \eqref{alphacond} holding for the set $E_0$ equal to the subgraph of $\varphi$ in $\Co \Omega$, and $\psi$ smooth on $\overline A$, then for all $k$ large enough there is  $s_k\in (0,1)$ such that 
	\[ u \leq -k \quad  \mbox{ almost everywhere in } \, \Omega \setminus A , \qquad u_s=\psi \quad \mbox{ in } A, \quad \mbox{ for all } s\leq s_k.\] Furthermore, if $A=\Omega$, then the solutions "sticks" to the obstacle, meaning that $u_s=\psi$ in $\Omega$.
	The stickiness is obvious at this point, since we can take $k $ larger than both the $L^\infty$ norm of $\varphi$ in some small neighborhood of $\Omega$, and the $L^\infty$ norm of $\psi$ in $A$. 
	\\ 
	The precise statements can be found as follows: the first item is contained in Theorem \ref{all}, the second statement in Theorem \ref{asympt_obst}. We add the factor $\varepsilon$ to underline the fact that this behavior is still apparent even for small obstacles, and we comment this in Remark \ref{rmkwitheps}.
	
	\medskip
	It is quite interesting to observe that the main result implies that the $s$-minimal function $u_s$ is not continuous across the boundary of the domain $\Omega$, nor across the domain of the obstacle. Furthermore, notice that this result does not depend on the regularity of the obstacle nor on the exterior data, nor on the geometry of the domain. In fact, by \cite[Theorem 1.1]{cabrecozzi} one obtains $u_s\in C^\infty(\Omega\setminus \overline A)$. On the other hand, even if one considers smooth data, i.e. $\varphi\in C^\infty(\Co \Omega)$ and $\psi\in C^\infty(\overline A)$, then by our main result, for $s$ small enough, the function $u_s$ cannot be globally continuous. Moreover, even without obstacles ($A=\emptyset$), the continuity of $u_s$ fails across $\partial \Omega$, independently of the geometry of the domain and of the regularity of the prescribed exterior data. This behavior is already known for the particular case in which the exterior data is a small perturbation of the plane (see \cite[Theorem 1.4]{boundary}). This phenomenon is purely nonlocal and it is strikingly different from the classical case (see the boundary regularity result \cite[15.9 Theorem]{Giusti}).

	\medskip
	
	In the rest of the paper, we do the following:  in Section \ref{one} we discuss the Plateau problem from a functional point of view. In Section \ref{setting} we discuss the geometric point of view and the equivalence with the functional setting. Section \ref{two} is dedicated to the stickiness and asymptotic properties as $s\to 0$. %The additional Section \ref{three} deals with the continuity of the obstacle problem in the fractional parameter. 
	We conclude the paper with an Appendix %\ref{dens} we prove a density result. Appendix 
	\ref{appendicite} in which we give an additional strategy for the proof of our stickiness result. 
	
	\medskip
	
	From here onward, the following notations will be used:
	\[X=(x,x_{n+1})\in \R^{n+1},\,  x\in \Rn,\]
	\[ \mathcal B_R(X)= \big\{ Y\in \R^{n+1} \; \big| \; |Y-X|<R\big\},\]
	\[ B_R(x)=\big\{ y\in \R^{n} \; \big| \; |y-x|<R\big\},\]
	\[ \omega_{n+1}=\Ha^n(\partial \mathcal B_1) = \frac{2\pi^{\frac{n+1}2}}{\Gamma \left(\frac{n+1}2\right)},\]
	\[\Omega^k:= \Omega\times (-k,k),\qquad \Omega^\infty:= \Omega\times\R.\]
	For any $\mathcal A \subset \Rn$, 
	we denote the subgraph of a function
	$u\colon \mathcal A \to \R$, , as 
	\bgs{ \Sg (u, \mathcal A) =\left\{ X= (x, x_{n+1})\in \mathcal  A\times \R \, \big |\, x_{n+1}< u(x)\right\},} and we write for simplicity 
	\bgs{ \Sg (u) :=\Sg(u, \Rn).} 
	%\left\{ X= (x, x_{n+1})\in \R^{n+1}\, \big |\, x_{n+1}< u(x)\right\},}
	\noindent Given an open set $\Omega\subset\R^n$ and $\delta\in\R$, the  signed distance function from $\partial\Omega$ is given by
	\[ \bar{d}_\Omega= d(x,\Omega)-d(x,\Co \Omega)\quad\mbox{for every }x\in\R^n, \qquad \mbox{ where } d(x,\Omega)=\inf_{y\in \Omega}|x-y|  \] 
	-- notice that the signed distance is negative inside $\Omega$. We also denote 
	\eqlab{ \label{neigh}\Omega_\delta:=\{x\in\R^n\,|\,\bar{d}_\Omega(x)<\delta\}.}
	\noindent For a bounded set with $C^2$ boundary, the signed distance is of class $C^2$ in a neighborhood of $\partial\Omega$ (see e.g. \cite{GilTru,Ambrosio}). More precisely, there exists $r_0>0$
	such that
	\[\bar{d}_\Omega\in C^2(N_{2r_0}(\partial\Omega)),\quad
	\mbox{ where }\quad N_{2r_0}(\partial\Omega):=\{x\in\R^n\,|\,\;|\bar{d}_\Omega(x)|<2r_0\}.\]
	As a consequence, since $|\nabla\bar{d}_\Omega|=1$,
	the open set $\Omega_\delta$ has $C^2$ boundary for every $|\delta|<2r_0$.
	For a more detailed discussion, see \cite{Doktor} or Appendix A.2 in \cite{BucurLombardini} and other references therein.
	
	We will also adopt in the paper the following notations for the measure theoretic interior and exterior of a measurable set $E\subset \R^{n+1}$:
	\eqlab{ \label{measth} 
		&\; E_{int}:=\{ x\in \R^{n+1} \, | \, |E\cap B_r(x)| =|B_r(x)|\mbox{ for some } r>0  \}, \\ &\; E_{ext} := \{ x\in \R^{n+1} \, | \, |E\cap B_r(x)| =0  \mbox{ for some } r>0  \}.}  When there is no risk of confusion, we will still use $\partial E, \overline E$ for the boundary, respectively the closure of the set a in a measure theoretic sense. 
	
	%%%%%%%%%%SECTION%%%%%%%%%%%%%

	\section{Functional setting for the Plateau problem with obstacles} \label{one}
	
	% \noindent \textbf{Nonlocal minimal graphs.}\\

	The purpose of this section is to 
	adapt the functional setting for the Plateau problem,  introduced in \cite{CozziLombardini} to the study of the \emph{obstacle problem}. To that end, we recall the necessary notions,  without aiming at full generality, neither in the statements
	nor in the proofs of our results.
	In particular, we consider only bounded obstacles and we
	prove the existence of a solution only  for exterior data which is
	bounded in a big enough neighborhood of the domain~$\Omega$.
	Furthermore, we will not investigate the regularity properties
	of such a solution or that of the free boundary. 
	
	%We consider the
	%Plateau problem with (eventually discontinuous) obstacles.
	%Namely, besides imposing the exterior data condition
	%$$u=\varphi\quad\mbox{a.e. in }\Co\Omega,$$
	%we constrain the functions to lie above a fixed function which acts as an obstacle, that is
	%$$u\ge\psi\quad\mbox{a.e. in }A,$$
	%where~$A\subset\Omega$ is a fixed open set.
	
	On these grounds, for every $t\in \R$ and $s\in (0,1)$, let
	\begin{align}
	\label{gdef}
	g_s(t) & := \frac{1}{(1+t^2)^\frac{n+1+s}{2}}, \\
	\label{Gdef}
	G_s(t) & := \int_0^tg_s(\tau)\,d\tau, \\
	\label{GGdef}
	\G_s(t) & := \int_0^tG_s(\tau)\,d\tau=\int_0^t\Big(\int_0^\tau g_s(r)\,dr\Big)d\tau.
	\end{align}

	Let $\Omega\subset\Rn$ be a bounded open set with Lipschitz boundary. For every $u:\Rn\to\R$, we define
	\eqlab{\label{Adef}
		\A_s(u,\Omega):=\int_\Omega\int_\Omega\G_s\Big(\frac{u(x)-u(y)}{|x-y|}\Big)\dkers,
	}
	\bgs{
		\Nl_s(u,\Omega):=2\int_\Omega\int_{\Co\Omega}\G_s\Big(\frac{u(x)-u(y)}{|x-y|}\Big)\frac{dx\,dy}{\kers},
	}
	and 
	\eqlab{\label{area}
		\Fc_s(u,\Omega):=\A_s(u,\Omega)+\Nl_s(u,\Omega)=\iint_{Q(\Omega)}\G_s\Big(\frac{u(x)-u(y)}{|x-y|}\Big)\frac{dx\,dy}{\kers},
	}
	where
	\bgs{
		Q(\Omega):=\R^{2n}\setminus(\Co\Omega)^2.
	}
	
	We point out that that function $\mathcal G_s$ is strictly convex, and there exist $\lambda, \Lambda>0$ such that 
	\eqlab{\label{lips}
		| \G_s(t)-\G_s(\tau) |\leq \Lambda |t-\tau|
	}
	(i.e. $\mathcal G_s$ is Lipschitz)
	and furthermore 
	\[ \Lambda |t| -\lambda \leq \mathcal G_s(t) \leq \Lambda |t|.\]
	
	We also denote
	\eqlab{ \label{Gbardef}
		\overline{G}_s(t):=\int_{-\infty}^tg_s(\tau)\,d\tau=\int_{-t}^{+\infty}g_s(\tau)\,d\tau.
	}
	Given $M\ge0$ %, we denote \[ \Omega^M=\Omega \times (-M,M) \qquad \Omega^{\infty} =\Omega \times \R.\]
	we define for every $u:\Rn\lra\R$
	\eqlab{ \label{NMdef}
		\Nl_s^M(u,\Omega):=\int_\Omega dx\bigg[\int_{\Co\Omega}dy\bigg(\int_{\frac{-M-u(y)}{|x-y|}}^\frac{u(x)-u(y)}{|x-y|}
		\overline{G}_s(t)\,dt+\int^\frac{M-u(y)}{|x-y|}_\frac{u(x)-u(y)}{|x-y|}
		\overline{G}_s(-t)\,dt\bigg)\frac{1}{\kers}\bigg]
	}
	and
	\eqlab{ \label{FMdef}
		\Fc^M_s(u,\Omega):=\A_s(u,\Omega)+\Nl_s^M(u,\Omega).
	}

	We introduce the functional spaces for  the problem \eqref{mainpb2}. We define
	\bgs{
		\W^s(\Omega):=&\left\{u:\Rn\lra\R\;\big|\;u|_\Omega\in W^{s,1}(\Omega)\right\},
	} 
	where
	\[ \|u\|_{W^{s,1}(\Omega)} = \|u\|_{L^1(\Omega)} + [u]_{W^{s,1}(\Omega)	},\]
	with
	\[ [u]_{W^{s,1}(\Omega)}=\int_{\Omega} \int_{\Omega} \frac{ |u(x)-u(y)|}{|x-y|^{n+s}} \, dx dy.	\]  
	Furthermore, we consider
	\bgs{
		\B\W^s(\Omega):=&\; \W^s(\Omega)\cap L^\infty(\Omega), 
		\\ 
		\B_M\W^s(\Omega):=&\left\{u\in\B\W^s(\Omega)\;\big|\;\|u\|_{L^\infty(\Omega)}\le M\right\},
	}
	for $M\ge0$. For $\varphi:\Rn \to\R$ , we denote
	\bgs{ 
		\W_\varphi^s(\Omega):=&\left\{u\in \W^s(\Omega) \;\big|\;u=\varphi \mbox{ on } \Co \Omega\right\},
		\\
		\B\W_\varphi^s(\Omega):= & \; \W_\varphi^s(\Omega) \cap L^\infty(\Omega),
		\\
		\B_M\W_\varphi^s(\Omega) :=  &\left\{u\in\B\W_\varphi^s(\Omega)\;\big|\;\|u\|_{L^\infty(\Omega)}\le M\right\}.
	}
	Let~$A\subset\Omega$ be an open set, $\psi\in L^\infty(A)$ be an obstacle function, and let $M\ge\|\psi\|_{L^\infty(A)}$. 
	We then introduce 
	\bgs{
		% \K^s(\Omega,\varphi,A,\psi)
		\W^s_{\varphi,\psi}(\Omega, A)
		:=&
		\{u\in\W^s_\varphi(\Omega)\,:\,u\ge\psi\mbox{ a.e. in }A\}
		,\\
		%\B\K^s(\Omega,\varphi,A,\psi):=
		\B\W^s_{\varphi,\psi}(\Omega, A)
		:=&
		\{u\in\B\W^s_\varphi(\Omega)\,:\,u\ge\psi\mbox{ a.e. in }A\},\\
		%\B_M\K^s(\Omega,\varphi,A,\psi)
		\B_M \W^s_{\varphi,\psi}(\Omega, A)
		:= &%\K^s(\Omega,\varphi,A,\psi)
		\;  \W^s_{\varphi,\psi}(\Omega, A)\cap\B_M\W^s(\Omega).
	}
	%}
	\begin{defn}\label{sol1}
		We say that a function~$u\colon \Rn \to \R$ solves the obstacle problem if~$u$ minimizes~$\F_s$
		in~$\W^s_{\varphi,\psi}(\Omega, A)$, i.e. if
		\eqlab{\label{minarea}
			\iint_{Q(\Omega)} \left\{ \G_s \left( \frac{u(x) - u(y)}{|x - y|} \right) - \G_s \left( \frac{v(x) - v(y)}{|x - y|} \right) \right\} \frac{dx\,dy}{|x - y|^{n - 1 + s}} \le 0,
		}
		for every~$v\in\W^s_{\varphi,\psi}(\Omega, A)$. 
	\end{defn}

	\begin{remark}\label{trunc} 
		We remark that this definition is well posed, thanks
		to \cite[Lemma 2.8]{CozziLombardini}.  \\
		Since $\psi \in L^\infty(A)$ given $u\in \W^s_{\varphi,\psi}(\Omega, A)$, for any $k\geq ||\psi\|_{L^\infty(A)}$ we can consider the truncations 
		\sys[u_k(x)=]{ &\varphi(x), &&\mbox{ for } \; x \in \Co \Omega
			\\
			& u(x), &&\mbox{ for } \; x \in \Omega, -k \leq u(x)\leq k 
			\\
			&-k, &&\mbox{ for } \; x \in \Omega, -k \geq u(x) 
			\\
			& k, &&\mbox{ for } \; x \in \Omega,  u(x)\geq k ,}
		and we have that 
		$u_k \in \B \W^s_{\varphi, \psi}(\Omega, A)$. Moreover 
		\[ \lim_{k\to +\infty} \|u_k -u\|_{W^{s,1}(\Omega)} =0.\]
		As a consequence, as observed in \cite[Lemma 2.8]{CozziLombardini} for all $M\geq 0$ it holds that
		\[ \lim_{k\to +\infty} \F^M_s(u_k, \Omega) = \F^M_s(u, \Omega).\]
		
		We point out also that \cite[Lemma 2.8]{CozziLombardini} gives an alternative definition for a solution of the obstacle problem: for some $M \geq 0$, $u$ solves the obstacle problem if and only if 
		\[ \F^M(u,\Omega) \leq  \F^M(v,\Omega) ,\]
		for all $v\in \W^s_{\varphi, \psi}(\Omega, A)$.
	\end{remark} 
	
	We define furthermore the ``tail'' of a function $u\colon \Rn \to \R$, restricted to a set $O \subset \Co \Omega$ at a point $x \in \Omega$ as
	\[ 
	\Tails(u, O, x)= \int_{ O} \frac{ |u(y)|}{|x-y|^{n+s}} \, dy.
	\]

	\subsection{Existence results} \label{functsett} 
	
	A first main result of this section is the following existence and uniqueness theorem.

	\begin{theorem}\label{Dirichlet_obstacle}
		Let~$n\ge1$,~$s\in(0,1)$,~$\Omega\subset\Rn$ a bounded open set with Lipschitz boundary. Let~$A\subset\Omega$ be an open set and~$\psi\in L^\infty(A)$.
		Then, there exists a constant $\Theta:=\Theta(n,s)>1$ such that for every $\varphi \colon \R^n \to \R$ with~$\varphi\in L^\infty(\Omega_{\Theta\diam(\Omega)  }\setminus\Omega)$, there exists a
		unique function $u\in\W^s_{\varphi,\psi}(\Omega, A)$ that solves the obstacle problem.
		Moreover
		\eqlab{\label{jbond}
			\|u\|_{L^\infty(\Omega)}\le \diam \Omega+\max\big \{ 
			\|\varphi\|_{L^\infty(\Omega_{\Theta\diam(\Omega)}\setminus\Omega)  },\|\psi\|_{L^\infty(A)}\big\}.
		}
	\end{theorem}
	It is interesting to observe that
	a solution exists without having to impose, besides boundedness, any regularity assumptions nor on the domain~$A$ where the obstacle is defined,
	nor on the obstacle function~$\psi$. 
	
	\begin{proof}
		%The argument is essentially the same one that we already employed to prove the existence of minimizers of~$\F$
		%in~$\W^s_\varphi(\Omega)$. 
		We begin by considering the functions~$u_M$
		that minimize~$\F_s^M(\,\cdot\,,\Omega)$ in~$\B_M\W^s_{\varphi,\psi}(\Omega, A)$, then we show that they stabilize. The proof follows the steps of \cite[Proposition 5.3]{CozziLombardini}  to prove the existence and uniqueness of a minimizer for $\F_s^M(\cdot,\Omega)$, and \cite[Proposition 3.5]{CozziLombardini} to show their stabilization as $M\to \infty$. 
		%by exploiting \textcolor{red}{\cite[Proposition 5.13]{}. 
		
		\subsubsection*{Step 1. Existence of $M$-minimizers.}
		Let $M\ge\|\psi\|_{L^\infty(A)}$. 
		Notice that \[ \F^M(\cdot,\Omega) \geq 0\] in $\B_M\W^s_{\varphi,\psi}(\Omega, A)$, then let
		\eqlab{\label{flu1} \inf\left\{\F^M(v,\Omega)\,:\,v\in\B_M\W^s_{\varphi,\psi}(\Omega, A)\right\}:=m,}
		and consider a minimizing sequence $u_M^{(k)}$.
		We point out that \eqref{flu1} gives a uniform bound on $[u_M^{(k)}]_{W^{s,1}(\Omega)}$ (see the bound  in formula \cite[(5.2)]{CozziLombardini})  hence on $\|u_M^{(k)}\|_{W^{s,1}(\Omega)}$ thanks to \cite[Lemma D.1.2]{lucaphd}.
		As a consequence of the lower semicontinuity for $\F_s^M$ in \cite[Lemma 5.1]{CozziLombardini} we obtain the existence of a minimizer $u_M\in \B_M\W_\varphi(\Omega)$, where up to a subsequence
		\[ \lim_{k \to \infty}  u_M^{(k)}= u_M  \qquad \mbox{ in } L^1(\Omega) \mbox{ and a.e.  in  } \Omega,\] 
		and 
		\[  \F_s^M(u_M,\Omega)=m.\]
		However, we observe that
		\bgs{
			\W^s_{\varphi,\psi}&(\Omega, A)\subset\W^s_\varphi(\Omega),
			\qquad\B\W^s_{\varphi,\psi}(\Omega, A)\subset\B\W^s_\varphi(\Omega)\\
			&
			\quad\mbox{and}\quad
			\B_M\W^s_{\varphi,\psi}(\Omega, A)\subset\B_M\W^s_\varphi(\Omega)
		}
		are closed convex subsets. This implies that $u_M\in\B_M\W^s_{\varphi,\psi}(\Omega, A) $ and it minimizes $\F_s^M$ in $\Omega$ within this space. The uniqueness of such a minimizer is ensured by the strict convexity of $\F_s^M$, \cite[Lemma 2.7]{CozziLombardini}.

		\subsubsection*{Step 2. Stabilization as $M\to \infty$.}
		
		Now we remark that, since the obstacle~$\psi$ is bounded, we can apply \cite[Proposition 3.5]{CozziLombardini}
		to obtain an a-priori bound on the~$L^\infty$ norm of the minimizers~$u_M$, provided~$M>0$ is big enough.
		Let
		\bgs{
			N:=\diam(\Omega)+\max\left\{\sup_{\Omega_{\Theta\diam(\Omega)}\setminus\Omega}\varphi, \, \sup_A\psi\right\},
		} 
		take
		\bgs{
			u_M^{(N)}:=\chi_\Omega\min\{u_M,N\}+(1-\chi_\Omega)u_M,
		}
		and notice that in $\Omega$,
		\[
		u_M^{(N)} \leq N.\]
		We point out that, for every $M\geq N$ if~$u_M\in\B_M\W^s_{\varphi,\psi}(\Omega, A)$ since~$N\ge\sup_A \psi$  then we clearly
		have~$u_M^{(N)}\in \B_M\W^s_{\varphi,\psi}(\Omega, A)$.
		According to \cite[Proposition 3.5]{CozziLombardini},  we have that
		\[ \F^M_s(u_M^{(N)}, \Omega) \leq  \F^M_s(u_M, \Omega).\] 
		By the uniqueness of the minimizer of the functional~$\F_s^M(\,\cdot\, ,\Omega)$
		in~$\B_M\W^s_{\varphi,\psi}(\Omega, A)$, we can conclude that
		\bgs{
			\sup_\Omega u_M\le N=\diam(\Omega)+\max\left\{\sup_{\Omega_{\Theta\diam(\Omega)}\setminus\Omega}\varphi, \sup_A\psi\right\},
		}
		for every~$M\ge N$.

		Since we can argue in the same way by truncating the functions from below, we find that
		\eqlab{\label{obst_pf_eqn1}
			\|u_M\|_{L^\infty(\Omega)}\le \diam(\Omega) +
			\max\left\{\|\varphi\|_{L^\infty(\Omega_{\Theta \diam(\Omega)}\setminus\Omega)},\|\psi\|_{L^\infty(A)}\right\}=:N_0,
		}
		for every~$M\ge N_0$.
		
		\subsubsection*{Step 3}
		Fix~$M_0:=N_0+1$ and observe that~\eqref{obst_pf_eqn1}
		gives that
		\eqlab{\label{obstacle_monster_power_mannaggia}
			\|u_{M_0}\|_{L^\infty(\Omega)}\le N_0<M_0.
		}
		We claim that this implies that the function~$u:=u_{M_0}$
		solves the obstacle problem. In order to prove this, let us consider~$v\in\B\W^s_{\varphi,\psi}(\Omega, A)$
		and notice that by~\eqref{obstacle_monster_power_mannaggia} we have
		\bgs{
			w:=\lambda v+(1-\lambda)u\in\B_{M_0}\W^s_{\varphi,\psi}(\Omega, A),
		}
		provided~$\lambda\in(0,1)$ is small enough. Thus, by the minimality of~$u$ and exploiting the convexity of~$\F^{M_0}$,
		we find
		\bgs{
			\F_s^{M_0}(u,\Omega)\le\F_s^{M_0}(w,\Omega)\le\lambda\F_s^{M_0}(v,\Omega)+(1-\lambda)\F_s^{M_0}(u,\Omega),
		}
		that is
		\bgs{
			\F_s^{M_0}(u,\Omega)\le\F_s^{M_0}(v,\Omega).
		}
		This shows that~$u$ minimizes~$\F^{M_0}_s(\,\cdot\,,\Omega)$ in~$\B\W^s_{\varphi,\psi}(\Omega, A)$.  As a consequence of \cite[Lemma 2.8]{CozziLombardini}, as noted in Remark \ref{trunc}, we obtain that  $u$ minimizes $\F_{M_0}^s(\cdot, \Omega)$  in the larger space~$\W^s_{\varphi,\psi}(\Omega, A)$, which implies that~$u$ solves the obstacle problem in the sense of the Definition \ref{sol1}. Finally, the strict convexity of~$\F_s^M$
		guarantees the uniqueness of such a solution, %---see also point~$(iii)$ of \cite[Remark 5.4]{CozziLombardini}---
		concluding
		the proof of the Theorem.
	\end{proof}
	
	We state a more general existence result, whose proof is a simple adaptation of the proof of \cite[Theorem 1.3]{CozziLombardini}.
	
	\begin{theorem}\label{Dirichlet_obstacle1}
		Let~$n\ge1$,~$s\in(0,1)$,~$\Omega\subset\Rn$ a bounded open set with Lipschitz boundary. 	Let~$A\subset\Omega$ be an open set and let~$\psi\in L^\infty(A)$. Then, there exists a constant $\Upsilon:=\Upsilon(n,s)>1$ such that for every $\varphi \colon \R^n \to \R$ with
		$
		\Tail_s(\varphi,\Omega_{\Upsilon \diam(\Omega)}\setminus \Omega;\, \cdot\,)\in L^1(\Omega)
		$
		there exists a
		unique function $u\in\W^s_{\varphi,\psi}(\Omega, A)$ that solves the obstacle problem.
		Moreover
		\bgs{
			\|u\|_{W^{s,1}(\Omega)}\le C\left(\|\Tail_s(\varphi, \Omega_{\Upsilon \diam(\Omega)}\setminus \Omega, \, \cdot \,)\|_{L
				^1(\Omega)}+\|\psi\|_{L^\infty(A)} +1\right),
		}
		where $C:=C(n,s,\Omega)>0$.
		
	\end{theorem}
	
	\begin{proof}
		Since the proof follows from  \cite[Theorem 1.3]{CozziLombardini} with some minor adjustments, we only give here a sketch. First of all, we denote
		\[
		%\tilde u= u- \sup_{A} |\psi|, \qquad 
		\tilde \varphi = \varphi - 
		\sup_{A} |\psi|, \quad \tilde \psi = \psi - \sup_{A} |\psi|.
		\]
			Reasoning as in Step 1 of Theorem \ref{Dirichlet_obstacle}, we  obtain the existence of a minimizer $u_M\in \B_M \W^s_{\tilde \varphi,\tilde \psi}(\Omega, A)$ of $\F^M$ and get a uniform bound on  
		$\|u_M\|_{W^{s,1}(\Omega)}$. More precisely, observing that $v=\chi_{\Co\Omega}u_M\in \B_M \W^s_{\tilde \varphi,\tilde \psi}(\Omega, A)$, a careful inspection of the proof of \cite[Proposition 3.2]{CozziLombardini}, and of the constants appearing therein, gives
		\eqlab{ \label{precc}
			\|u_M\|_{W^{s,1}(\Omega)}	 \leq &\;
			\frac{c_n}{1-s} |\Omega| d \left[ \left(1+\frac{1}s\right)d^{-s} +1\right]\\ &\;  + c_n \| \Tails(\tilde \varphi, \Omega_{{\Upsilon} \diam(\Omega)}\setminus \Omega, \cdot) \|_{L^1(\Omega)} \left(d^s+\frac{1}s+1\right) ,
		}
		where he denote for simplicity $d=\diam(\Omega)$.
		By compactness we can extract a convergent subsequence, which converges to the minimizer $\tilde u\in\W^s_{\tilde\varphi,\tilde\psi}(\Omega, A)$ and the bound \eqref{precc} holds also for $\tilde u$ .  	
		Now taking 
		\[
		u=\tilde u + \sup_A|\psi|,
		\]
		we get that $u\in \W^s_{\varphi,\psi}(\Omega, A)$  is the solution of the obstacle problem in the sense of the Definition \ref{sol1} and
		\eqlab{\label{unifst}
			\| u\|_{W^{s,1}(\Omega)} 
			\leq &\;\frac{c_n}{1-s} |\Omega| d \left[ \left(1+\frac{1}s\right)d^{-s} +1\right] + c_n \| \Tails( \varphi, \Omega_{{\Upsilon} \diam(\Omega)}\setminus \Omega, \cdot) \|_{L^1(\Omega)} \left(d^s+\frac{1}s+1\right)\\ &\; +  \|\psi\|_{L^\infty(A)} \Per_s(\Omega, \Rn) \left[1+ c_n\left(d^s+\frac{1}s+1\right) \right].
		}
		From this, the conclusion is settled.
	\end{proof}

	\subsection{First variation of the fractional area functional} \label{funcgeom}

	In this subsection, we deal with the Euler-Lagrange equation for the area functional, in particular with its weak formulation. 
	
	We define 
	\[ \h_s u(x):= 2 \mbox{\, P.V.} \int_{\Rn} G_s\left(\frac{u(x)-u(y)}{|x-y|}\right) \frac{dy }{|x-y|^{n+s}},\]
	noticing that $u\colon \Rn \to \R$ is required to be $C^{1,\alpha}$ for some $\alpha>s$ around $x$, in order for $\h u$ to be defined point-wisely. Notice the connection between $H_s$ and $\h$: when $E$ is a subgraph, the mean curvature can be written as 
	\[ H_s[\Sg(u)]((x,u(x)))= \h_s u(x),\]
	see for instance \cite{regularity,bootstrap,BucurLombardini,CozziLombardini}. 
	
	Notice that without regularity assumptions on the function $u$, a weak framework can be introduced, according to the following definition.
	
	\begin{defn}
		Let $\Omega \subset \Rn$ be an open set, a function
		$u\colon \Rn \to \R$ is a weak subsolution of $\h  u =0$ in $\Omega$ and we write $\h u \leq 0$  if for all $v\in C^\infty_c(\Omega)$ with $v\geq 0$, it holds that
		\[ \langle \h_s u, v\rangle:= \int_{\R^{2n}} G_s\left(\frac{u(x)-u(y)}{|x-y|}\right) (v(x)-v(y)) \frac{dx\, dy }{|x-y|^{n+s}} \leq 0.\]
		The function $u$ is a weak subsolution and we write $\h_s u \geq 0$ if $\h (-u) \leq 0,$  and it is a weak solution if $u$ is both a super and a subsolution of $\h_s u =0$ in $\Omega$.
	\end{defn}
	
	We point out that a solution of the obstacle problem is a supersolution of the equation~$\h_s u=0$
	in the whole domain~$\Omega$ and a solution away from the contact set, that is, formally:
\bgs{
		\h_s u\ge0\quad\mbox{ in }\Omega\quad\mbox{and}\quad
		\h_s u=0\quad\mbox{ in }\Omega\setminus\{u=\psi\}.
	}

	More precisely, we have the following result:
	\begin{prop}\label{eulero}
		Let~$n\ge1$,~$s\in(0,1)$,~$\Omega\subset\R^n$ a bounded open set with Lipschitz
		boundary,~$\varphi \colon \Rn \to\R$,~$A\subset\Omega$ an open set and~$\psi\in L^\infty(A)$.
		Then a function~$u\in\W^s_{\varphi,\psi}(\Omega, A)$ solves the obstacle problem if and only if
		\eqlab{\label{equa} \langle\h_s u,v-u\rangle\ge0\quad \mbox{ for all } \,v\in \W^s_{\varphi,\psi}(\Omega, A). }
		In particular, if there exists a function~$u\in\W^s_{\varphi,\psi}(\Omega, A)$ that solves the obstacle problem: 
		\begin{enumerate}[(i)]
			\item  it holds that \bgs{
				\langle\h_s u,w\rangle\ge0\quad \mbox{ for all } \,w\in C_c^\infty(\Omega)\quad\mbox{s.t. }w\ge0.
			}
			\item if~$\Op\subset\Omega$ is an open set such that
			\bgs{
				\inf_{\Op\cap A} (u-\psi)\ge\delta,
			}
			for some~$\delta>0$, then
			\bgs{
				\langle\h_s u,w\rangle=0\quad\forall\,w\in C^\infty_c(\Op).
			}
			Furthermore, if~$\Op$ has Lipschitz boundary, then~$u$ minimizes~$\F_s$ in~$\W^s_u(\Op)$.
		\end{enumerate}

		%Moreover,~$u\in L^\infty_{\loc}(\Omega)$.
	\end{prop}
	
	\begin{proof}
		Given any $u,v \in \W^s_{\varphi,\psi}(\Omega, A)$ by convexity we get that
		\[\F_s^0(v, \Omega)- \F_s^0(u, \Omega) \geq \langle \h_s u, v-u\rangle. \]
		Indeed, using \cite[Lemma 2.8]{CozziLombardini} and exploiting the convexity of $\mathcal G_s$,
		\bgs{ \F_s^0(v, \Omega)- \F_s^0(u, \Omega)= &\; \iint_{Q(\Omega)} \left[ \mathcal G_s\left(\frac{v(x)-v(y)}{|x-y|} \right) -\G_s \left(\frac{u(x)-u(y)}{|x-y|} \right)\right] \frac{dx \, dy}{ |x-y|^{n-1+s}}\\
			\geq &\; \iint_{Q(\Omega)}  G_s\left(\frac{u(x)-u(y)}{|x-y|} \right) \frac{ (v-u)(x)-(v-u)(y)}{|x-y|} \frac{dx \, dy}{ |x-y|^{n-1+s} }
			\\
			=&\; \iint_{\R^{2n}}  G_s\left(\frac{u(x)-u(y)}{|x-y|} \right) \left( (v-u)(x)-(v-u)(y)\right) \frac{dx \, dy}{ |x-y|^{n+s} }
			\\
			=&\; \langle \h_su, v-u\rangle ,
		}
		using also the fact that $u-v=0$ on $\Co \Omega$. 
		Thus, if $u\in \W^s_{\varphi,\psi}(\Omega, A) $ satisfies \eqref{equa} then $u$ is a minimizer of $\F_s^0(\cdot, \Omega)$, and thanks to Remark \ref{trunc}, $u$ solves the obstacle problem. \\
		For the other implication,  since $u +t(v-u)$ is in $\W^s_{\varphi,\psi}(\Omega, A)$ for all $t\in [0,1]$, by Remark \ref{trunc} 
		\[ \frac{\F_s^0(u +t(v-u), \Omega) - \F_s^0(u,\Omega)}{t} \geq 0\] 
		thus by \cite[Lemma 2.16]{CozziLombardini}
		since $v-u\in \W^s_0(\Omega)$  we obtain that
		\[ \langle \h_su , v-u \rangle  =\lim_{t\to 0^+} \frac{\F_s^0(u +t(v-u), \Omega) - \F_s^0(u,\Omega)}{t} \geq 0, \]
		establishing \eqref{equa}.
		
		\smallskip 
		To prove (i) it is enough to observe that  for any  $w\in C^\infty_c(\Omega)$ then $v=  u + w \in \W^s_{\varphi, \psi}(\Omega, A)$ and use \eqref{equa}.

		In order to prove (ii), that~$u$ is a solution away from the contact set, let~$w\in C^\infty_c(\Op)$ and
		observe that for every~$|\eps|\le \delta/\|w\|_{L^\infty(\Op)}$ we have~$u+\eps w\in\W^s_{\varphi,\psi}(\Omega, A)$.
		Roughly speaking, since we are away from the contact set, we are allowed to deform the function~$u$ both from above and from below. Hence, again by the minimality of~$u$ and exploiting \cite[Lemma 2.16]{CozziLombardini}, we obtain~$\langle\h u,w\rangle=0$.
		
		Finally, if~$\Op$ has Lipschitz boundary, then we conclude that~$u$ minimizes~$\F$
		in~$\W^s_u(\Op)$ by \cite[Corollary 2.22]{CozziLombardini}.
	\end{proof}
	
	\begin{remark}
		If~$A\subset\subset\Omega$ has Lipschitz boundary,
		then~$u$ minimizes~$\F_s$ in~$\W^s_u(\Omega\setminus\overline{A})$.
	\end{remark}

	\subsection{Boundedness of solutions}
	In this subsection, we provide some a-priori $L^\infty$ bounds for solutions of the obstacle problem. \\
	
	We observe that if $u\in \W^s_{\varphi,\psi}(\Omega, A) $ solves the obstacle problem in $\Omega$ for exterior data $\varphi$ and obstacle $\psi$ in $A$, and $\Omega' \subset \Omega$ is an open set with Lipschitz boundary, then $u$  solves the obstacle problem in $\Omega'$ for exterior data $u$ and obstacle $\psi|_{\Omega'\cap A}$ in $\Omega'\cap A$. Precisely,  if we define $\tilde \varphi:= u$ and $\tilde \psi := \psi|_{\Omega'\cap A}$ then $u\in \W^s_{\tilde \varphi,\tilde \psi}(\Omega', \Omega'\cap A) $ solves the corresponding obstacle problem. \\
	When dealing with local bounds on the solution, it is thus enough to consider the case in which $\Omega=B_{2R}$ and $A \subset B_{2R}$. We  have the following results.
	
	\begin{theorem}\label{inftyloc} Let $\varphi \colon \Rn \to \R$ be a measurable function, $R>0$, $A\subset B_{2R}$ and $\psi \in L^\infty(A)$.
		If $u\in \W^s_{\varphi,\psi}(B_{2R}, A) $ solves the obstacle problem in $B_{2R}$, then
		\[ \sup_{B_R} u \leq  \max\left\{ C\left(R+ \dashint_{B_{2R}} u_+(x) \, dx\right), \sup_{ A} \psi_+\right\} ,\]
		for some constant $C:=C(n,s)>0$.
	\end{theorem}
	
	\begin{proof}
		The argument for the bound on the supremum is basically the same of \cite[Proposition 3.3]{CozziLombardini}. 
		The starting point, avoiding technical details, consists in considering  $\eta \in C_c^\infty(\Rn)$ a cutoff function with support in $B_{\tau}$ for $\tau< 2R$, laying between $0$ and $1$ (and satisfying some other properties) and taking
		$k \geq \sup_{A} \psi_+$. 
		Let \[ v=u- \eta w_k, \qquad \mbox{ with } \quad w_k:=(u-k)_+.\] It is clear that $v\leq u$ and $v\in \W^s_{u, \psi|_{B_\tau\cap A}}(B_{\tau}, A \cap B_{\tau})$  is a competitor for $u$ in $B_\tau$ with obstacle $\psi$ in $A \cap B_{\tau}$. We can thus compare the energies of $u$ and $v$ and follow the argument in \cite[Proposition 3.3]{CozziLombardini}, concluding that the set $\{ u\geq k\}\cap B_{R}$  has Lebesgue measure zero, provided $k$ is large enough.
	\end{proof}
	
	Notice that if $u$ solves the obstacle problem in $B_{2R}$, and $A=B_{2R}$ then $u$ has to stay above the obstacle, so $\inf_{A} u \geq \inf_{A} \psi$. However, using an argument similar to (the above), we might improve the estimate. The bound also holds where the obstacle is not defined. 
	\begin{theorem}\label{infb}
		Let $\varphi \colon \Rn \to \R$ be a measurable function, $R>0$, $A\subset B_{2R}$ and $\psi \in L^\infty(A)$.
		If $u\in \W^s_{\varphi,\psi}(B_{2R}, A) $ solves the obstacle problem in $B_{2R}$, then
		\[ \inf_{B_{R}}  u\geq  -  \min \left\{ C\left(R+ \dashint_{B_{2R}} u_-(x) \, dx\right), \sup_{ A} \psi_-\right\} , \]
		for some constant $C:=C(n,s)>0$.
	\end{theorem}
	%\[= max \left\{ -C\left(R+ \dashint_{B_{2R}} u_-(x) \, dx\right), \inf_{B_{2R}\cap A} (-\psi_-)\right\} \]
	\begin{proof}
		The proof follows the lines of Theorem \ref{inftyloc} (and the already mentioned \cite[Proposition 3.3]{CozziLombardini}. We consider $\tau <2R$ and $\eta\in C^\infty_c(\Rn)$ as defined in Theorem \ref{inftyloc}, take any $k \geq \sup_{A} \psi_-$ and let 
		\[ w_k:= (u+k)_-, \qquad \mbox{ and } \quad 
		v:= u-  \eta w_k.\]
		We notice that
		$  v \geq u $  in $B_{2R}$,  
		in particular  
		$ v \geq \psi  $  in $A$ 
		and  $ v\in \W^s_{u, \psi|_{A\cap B\tau}}(B_\tau, A\cap B_\tau) $  
		and again $v$ is a competitor for $u$ in $B_\tau$ with obstacle $\psi$ in $A \cap B_\tau$. The conclusion follows by showing that the set $\{u\leq -k\}\cap B_R$ has Lebesgue measure zero for $k$ large enough. 
	\end{proof}
	
	Of course, the above results yield  
	a bound on the norm $L^\infty(\Omega')$  for any $\Omega' \subset \subset \Omega$, by covering $\overline \Omega'$ with a finite number of balls with radius smaller than $\dist(\Omega', \partial \Omega) /2$ and using Theorems \ref{inftyloc} and \ref{infb} for each ball. Recalling the notation in \eqref{neigh}, since for any $\tau>0$ we have $\Omega_{-\tau} \subset \subset \Omega$, then for all $\tau >0$ the norm $\|u\|_{L^\infty(\Omega_{-\tau})}$ is bounded in terms of $\Omega, \|\psi\|_{L^\infty(A)}$ and $\|u\|_{L^1(\Omega)}$. We obtain furthermore an  estimate near the boundary. 
	
	\begin{prop} \label{nui}
		Let~$n\ge1$,~$s\in(0,1)$,~$\Omega\subset\Rn$ be a bounded open set with Lipschitz boundary.
		Let~$A\subset\Omega$ be an open set,~$\psi\in L^\infty(A)$ and let~$\varphi:\Rn\to\R$ be measurable function.
		If~$u\in\W^s_{\varphi,\psi}(\Omega, A)$ solves the the obstacle problem, and $\varphi \in L^\infty(\Omega_d\setminus \Omega)$ for some $d>0$, then $u\in L^\infty(\Omega)$ and it holds that
		\[ \|u\|_{L^\infty(\Omega\setminus \Omega_{-\theta d})} \leq  C\max\{ \|u\|_{L^\infty(\Omega_{-\theta d})} , \|\psi\|_{L^\infty(A)}, \|\varphi\|_{L^\infty(\Omega_d\setminus \Omega)}\}\]
		for some constant $\theta :=\theta(n,s,\Omega)\in (0,1)$.
	\end{prop}
	The proof follows the strategy in  \cite[Proposition 3.6]{CozziLombardini}, with small modifications. We insert here a sketch for completeness.  One considers
	\[ N \geq\ d+ \max\{  \sup_{\Omega_{-\theta d}} u, \sup_{\Omega_d\setminus \Omega} \varphi , \sup_A \psi \}  \]
	with $\theta$ small to be chosen later. Then one calls $\Omega_+ :=\{u >N\}$ and  proves that $|\Omega_+|=0.$
	\\
	To do this, using \cite[Proposition 3.5]{CozziLombardini} one defines $u^{(N)}$  as 
	\sys[ u^{(N)} = ]{ & \min\{ u, N \}  &&\mbox{ in } \Omega,\\
		& u &&\mbox{ in } \Co \Omega, }  
	and uses that 
	\[ \mathcal N_s^M(u^{(N)},\Omega) \leq \mathcal N_s^M(u, \Omega)\]
	to obtain  
	$  \beta_1+\beta_2 \geq 0$. 
	As in \cite[Proposition 3.6]{CozziLombardini}, we have denoted
	\[ \beta_1:=\int_{\Omega_+} \left\{ \int_{\Co B_d(x) \setminus \Omega}  \left[ \G \left( \frac{N-u(y)}{|x-y|} \right) -\G \left( \frac{u(x)-u(y)}{|x-y|} \right)  \right]  \, \frac{dy}{|x-y|^{n-1+s}} \right\} \, dx\]
	and
	\[\beta_2 := \int_{\Omega_+} \left\{  \int_{B_d(x) \setminus \Omega} \left[ \G \left( \frac{N-u(y)}{|x-y|} \right) -\G \left( \frac{u(x)-u(y)}{|x-y|} \right)  \right]  \, \frac{dy}{|x-y|^{n-1+s}} \right\} \, dx.\] 
	The same estimates in \cite[Proposition 3.6]{CozziLombardini} can be obtained and one gets that 
	\[ \frac{1}{d^s} \left(\frac{c_2}{\theta^s} -C_1\right) \int_{\Omega}(u(x)-N)_+\, dx  \geq 0,\]
	which by choosing $\theta$ small enough would give that $u\leq N$ almost anywhere in $\Omega$. This concludes the proof. 
	
	\section{Geometric obstacle problem}\label{setting}

	In this section we study the \emph{obstacle problem} in the geometric setting, i.e. we look for minimizers of the fractional perimeter in the unbounded domain~$\Omega^\infty$, in the class of subgraphs.
	This problem has been considered for general sets and in the case of bounded domains in~\cite{CDSS16}, where
	the authors proved a regularity result for the solution.
	
	We introduce in this subsection the general geometric setting, used to deal with any minimal set (and not only with subgraphs). The result obtained in this paper have a geometrical flavor, and this is mainly the reason why the functional setting is used only to prove existence and uniqueness of solutions. The two settings however are equivalent, see \cite[Theorem 1.10]{CozziLombardini}: minimizing the area functional introduced in Section \ref{functsett} (check \eqref{minarea}) for  some $u\in \W^{s,1}(\Omega)$, is equivalent to locally minimizing the perimeter in $ \Omega\times \R$ among sets with given exterior data $\Sg(u,\Co \Omega)$. The same equivalence holds also for the obstacle problem, as we see in Proposition \ref{equiv}.

	In this section, we do not aim at full generality. In particular, we define
	a geometric minimizer only in a setting of our interest, by considering only a bounded open set~$\Omega\subset\R^n$ with Lipschitz boundary, and
	bounded obstacles.
	
	Let~$n\ge1$,~$s\in(0,1)$, $\mathcal O\subset\R^{n+1}$ be a bounded open set with Lipschitz
	boundary.  The fractional perimeter defined as
	\eqlab{ \label{perimeter} \Per_s(E,\mathcal O)= \iint_{\R^{2(n+1)} \setminus (\Co \mathcal O)^2 } \frac{ |\chi_E(X)-\chi_E(Y)|}{|X-Y|^{n+1+s}}\, dXdY,} can be divided in
	\begin{equation} \label{nmspf1} \text{Per}_s(E,\mathcal O) := \Ll_s(E\cap \mathcal O,\Co E)
	+\Ll_s(E\setminus \mathcal O,\mathcal O \setminus E), \end{equation}
	where the interaction $\Ll_s(A,B)$ occurs between two disjoint subsets of $\R^{n+1}$, precisely
	\begin{equation}
	\Ll_s(A,B):=\int_A \int_B \frac{dX\, dY}{|X-Y|^{n+1+s}}
	=\int_{\R^{n+1}} \int_{\R^{n+1}} \frac{ \chi_A(X) \chi_B (Y) }{|X-Y|^{n+1+s}}\, dX \, dY.
	\end{equation}

	Let now~$\varphi:\R^n\to\R$,~$A\subset\Omega$ be an open set and let~$\psi\in L^\infty(A)$.	
	We say that a set~$E\subset\R^{n+1}$, given such that~$E\setminus\Omega^\infty=\Sg(\varphi)\setminus\Omega^\infty$
	and such that~$\Sg(\psi, A)\subset E$, \emph{solves the geometric obstacle problem}
	if for every~$M\ge\|\psi\|_{L^\infty(A)}$ it holds that~$\Per_s(E,\Omega^M)<\infty$, and if for every~$F\subset\R^{n+1}$ such that~$F\setminus\Omega^M=E\setminus\Omega^M$
	and~$\Sg(\psi,A)\subset F$, it holds 
	\bgs{
		\Per_s(E,\Omega^M)\le\Per_s(F,\Omega^M).
	}

	We recall that the $s$-fractional mean curvature of a set $E\subset \R^{n+1}$ at a point $Q\in\partial E$ is defined as the principal value integral
	\[H_s[E](Q):=P.V.\int_{\R^{n+1}}\frac{\chi_{\Co E}(Y)-\chi_E(Y)}{|Y-Q|^{n+1+s}}\,dY,\]
	precisely,
	\[H_s[E](Q):=\lim_{\rho\to0^+}H_s^\rho[E](Q),\qquad\textrm{where}\qquad
	H_s^\rho[E](Q):=\int_{
		\Co B_\rho(Q)}\frac{\chi_{\Co E}(Y)-\chi_E(Y)}{|Y-Q|^{n+1+s}}\,dY.\] 
	For the main properties of the fractional mean curvature, we refer e.g. to \cite{Abaty}.
	
	\begin{remark}\label{geom_obst_locmin_remark}
		We observe that if~$E\subset\R^{n+1}$ solves the geometric obstacle problem, then it is locally $s$-minimal in the open
		set~$\Omega^\infty\setminus\overline{\Sg(\psi,A)}$. Then we have an Euler-Lagrange equation along $\partial E$ in three cases: in the interior of the domain away from the obstacle, on the boundary of the domain and on the contact set. We write this precisely. 
		\begin{enumerate}
			\item 
			According to \cite[Theorem 5.1]{nms}, the following equation 
			\[ H_s[E]=0 \qquad \mbox{ in }   \partial E \cap \left( \Omega^\infty\setminus\overline{\Sg(\psi,A)} \right)\] is satisfied in the viscosity sense. Precisely, if $X_0\in \partial E \cap \left( \Omega^\infty\setminus\overline{\Sg(\psi,A)} \right)$ and $E$ has either an interior or exterior tangent ball at $X_0$, then $H_s[E](X_0)$ is well defined and is equal to zero (see \cite[Theorem B.7]{BucurLombardini}. 
			\item Assume that  $A\subset \subset \Omega$ and $\partial \Omega$ is of class $C^2$. Suppose that $X_0=(x_0,t_0)\in \partial E \cap \partial \Omega^\infty$, i.e. $x_0\in \partial \Omega$. If there exists $R>0$ such that 
			\begin{itemize}
				\item $\varphi(x)\leq t_0-R$ for almost every $x\in B_R(x_0)\setminus \Omega$, then 
				\[ H_s[E](X_0) \leq 0;\]
				\item $\varphi(x)\geq t_0+R$ for almost every $x\in B_R(x_0)\setminus \Omega$, then 
				\[ H_s[E](X_0) \geq 0.\]
				For more details, see \cite[Theorem B9]{BucurLombardini}.
			\end{itemize}
			Assume now that $A=\Omega$ again with $C^2$ boundary, and $\psi \in C(\overline A)$ then the same holds true provided that $\psi(x_0) <t_0$. 
			\item Consider $A\subset \Omega$ an open set and $\psi \in C(\overline A)\cap C^2(A)$. Let $X_0=(x_0,\psi(x_0)) \in \partial E$ with $x_0\in A$. Then $H_s[E](X_0)$ is well defined and
			\[ H_s[E](X_0) \geq 0.\]
			This follows from \cite[(2.2)]{CDSS16} (keep in mind that in this paper the fractional mean curvature is considered with the opposite sign), see also \cite[Theorem B.9]{BucurLombardini}.
		\end{enumerate}
		
		Moreover, we also observe that taking $A\subset \subset \Omega$ and $\partial A$ of class $C^2$, $\psi \in L^\infty(A)$ and $X_0=(x_0,t_0)\in \partial E$ with $x_0\in \partial A$ and assume that there exists $R>0$ such that $\psi(x)\geq t_0+R$ for almost any $x\in B_R(x_0)\cap A$. Then $H_s[E](X_0)$ is well defined and
		\[ H_s[E](X_0) \geq 0.\]
	\end{remark}

	By exploiting \cite[Theorem 1.10]{CozziLombardini}, 
	it is readily seen that if~$u$ solves the obstacle problem \eqref{mainpb2}, then its subgraph
	solves the geometric obstacle problem. 
	\begin{prop}\label{equiv}
		Let~$n\ge1$,~$s\in(0,1)$,~$\Omega\subset\Rn$ be a bounded open set with Lipschitz boundary %~$R_0>1$ be such that $\Omega\subset B_{R_0}$ 
		and let~$\Theta=\Theta(n,s)>1$ be as in \cite[Theorem 1.4]{CozziLombardini}.
		Let~$A\subset\Omega$ be an open set,~$\psi\in L^\infty( A)%C^0(\overline A)
		$
		and let~$\varphi:\R^n\to\R$ such that~$\varphi\in L^\infty(\Omega_{\Theta \diam(\Omega)}%B_{\Theta R_0}
		\setminus\Omega)$.
		Let~$u\in\W^s_{\varphi,\psi}(\Omega, A)$ be the unique solution of the obstacle problem, as in
		Theorem~\ref{Dirichlet_obstacle}. Then~$\Sg(u)$ solves the geometric obstacle problem.
	\end{prop}
	\begin{proof}
		The proof argument is basically the same as \cite[Theorem 1.11]{CozziLombardini}.\\ We consider  $M\geq \max\{ \|u\|_\infty(\Omega),  \|\psi\|_{L^\infty(A)}\} $ %based on truncation at height $k$.
		%M\geq \|u\|_\infty,  
		we first observe that $\Sg(u)$ is such that $\Sg(u)\setminus \Omega^\infty = \Sg(\varphi)\setminus \Omega^\infty,$  and $ \Sg(\psi, A) \subset \Sg(u)$. Moreover, 
		$\Per_s(\Sg(u) , \Omega^M)<\infty$, see \cite[Proposition 2.12]{CozziLombardini}, so that $E=\Sg(u)$ is an admissible candidate as a solution to the geometric obstacle problem.  
		
		We consider a competitor $F\subset\R^{n+1}$ such that~$F\setminus\Omega^M=\Sg(u) \setminus\Omega^M$
		and~$\Sg(\psi,A)\subset F$. Without loss of generality, we can further assume that $\Per_s(F,\Omega^M)<\infty$. Our choice of $M$ ensures that \cite[(1.16)]{CozziLombardini} is satisfied, i.e.
		\eqlab{\label{bmm} \Omega \times (-\infty, -M) \subset F \cap \Omega^\infty \subset \Omega \times (-\infty, M).} We can thus apply \cite[Theorem 1.9]{CozziLombardini}, rearranging $F$ in the vertical direction to obtain $\Sg(w_F)$ with 
		\[ w_F (x) = \lim_{R\to +\infty} \left(\int_{-R}^R \chi_{F}(x,t) \, dt -R\right),\]
		obtaining that
		\[ \Per_s(F, \Omega^M) \geq \Per_s(\Sg(w_F), \Omega^M).\]
		We also observe that $w_F\in \B_M \W^s_{\varphi, \psi}(\Omega, A)$, precisely $\|w_F\|_{L^\infty(\Omega)} \leq M$ follows from \eqref{bmm}, \cite[Proposition 2.12]{CozziLombardini} yields $w_F|_\Omega \in W^{s,1}(\Omega)$, $w_F=\varphi$ outside of $\Omega$ since $F\setminus \Omega^M = \Sg(u) \setminus \Omega^M$  and finally $\Sg(\psi, A) \subset F$ implies that $w_F\geq \psi$ in $A$. \\
		Since $u$ is a solution to the obstacle problem, we obtain the conclusion of the theorem using \cite[Proposition 2.12]{CozziLombardini}. Indeed,
		\[ \Per_s(\Sg (u), \Omega^M)= \F_s^M(u, \Omega) + \kappa_{\Omega, M} \leq \F_s^M(w_F, \Omega) + \kappa_{\Omega, M}  = \Per_s(\Sg(w_F), \Omega^M). \qedhere\]
	\end{proof}
	
	\begin{remark} We point out that the conclusion of Proposition \ref{equiv} holds under more general assumptions: if $u\in \W^s_{\varphi, \psi}(\Omega, A)$ solves the obstacle problem and $u\in L^\infty(\Omega)$ then $\Sg(u)$ solves the geometric obstacle problem. The existence and boundedness of a function $u$ solving the obstacle problem are ensured by, example given, Theorem \ref{Dirichlet_obstacle} and Proposition \ref{nui}. In particular, instead of requiring $\varphi$ to be bounded in $ \Omega_{\Theta \diam (\Omega)} \setminus\Omega$ we can require  that $\varphi$ has integrable tail, i.e. $\Tail_s(u, \Omega_{\Upsilon \diam(\Omega)} \setminus \Omega, \cdot) \in L^1(\Omega)$, and $\varphi \in L^\infty(\Omega_d\setminus \Omega)$ for some $d>0$.
	\end{remark}
	
	We note that the subgraph of~$u$ is actually the unique solution to the geometric obstacle problem.
	
	\begin{prop}
		\label{equiv1}	Let~$n\ge1$,~$s\in(0,1)$,~$\Omega\subset\R^n$ be a bounded open set with~$C^2$ boundary.
		Let~$A\subset\Omega$ be an open set  such that either~$A=\Omega$ or~$A\subset\subset\Omega$ with $C^2$ boundary. Let ~$\psi\in L^\infty(A)$ and~$\varphi:\R^n\to\R$  be such that~$\varphi\in %L^\infty(B_{R_s}\setminus\Omega)
		L^\infty_{loc}(\Rn)$.
		%with~$R_s$ as defined above. 
		Let~$u\in\W^s_{\varphi,\psi}(\Omega, A)$ be the unique solution of the obstacle problem, as in
		Theorem~\ref{Dirichlet_obstacle}. 
		Then~$\Sg(u)$ is the unique solution of the geometric obstacle problem.
	\end{prop}
	\begin{proof}
		By Proposition \ref{equiv} we have that $\Sg(u)$ solves the geometric obstacle problem. Let $E\subset\R^{n+1}$ be a solution of the geometric obstacle problem. 
		
		We claim that there exists $M_0=M_0(n,s,\Omega,\varphi,A,\psi)>0$ such that 
		\eqlab{ \label{bdd}
			\Omega\times(-\infty,-M_0)\subset E\cap\Omega^\infty
			\subset\Omega\times(-\infty,M_0).
		}
		Once this holds, we use \cite[Theorem 1.9]{CozziLombardini} to conclude that $E=\Sg(w_E)$ and $w_E\in \B \W^s_{\varphi, \psi}(\Omega, A) $ with $\|w_E\|_{L^\infty(\Omega) }\leq M_0$. Let $M \geq \max\{M_0, \|u\|_{L^\infty(\Omega)}\}$ then by minimality $E=\Sg(w_E)$ and using \cite[Proposition 2.12]{CozziLombardini} 
		\bgs{ \F_s^M(w_E,\Omega) = \Per_s(E, \Omega^M) - k_{\Omega,M} \leq \Per_s(\Sg(u),\Omega^M) - k_{\Omega,M} =\F_s^M(u, \Omega).
		}
		Since $u$ is the only minimizer of $\F_s^M$ in $\W^s_{\varphi, \psi}(\Omega, A)$, we conclude that $w_E=u$ almost everywhere in $\Rn$, hence $E=\Sg(u)$. 
		
		To prove \eqref{bdd} we can use the same argument of \cite[Lemma 3.3]{graph}: therein the authors begin with a ball $B$ which sits above (or below) $\varphi$, outside of $\Omega^\infty$, with radius $R$ large enough and center $X=(x, t)$ for $x\in \Co \Omega$ for all $t\geq t_0$ (or $t\leq - t_0$) for some $t_0$ large enough (depending on the exterior datum, $R, s$ and $\Omega$). Then they slide the ball horizontally until they first reach a contact point with $\partial E \cap \left(\overline \Omega\times \R\right)$. \\
		We can do the same by taking care of staying above (or below) the level $\|\psi\|_{L^\infty(\A)}$ (or $-\|\psi\|_{L^\infty(\A)}$) while sliding the ball. We thus need to take $t_0$ greater also than $\|\psi\|_{L^\infty(A)} +2R$. 
		
		At the contact point $X_0\in \partial E \cap \left(\overline \Omega \times \R\right)$, they prove that for $R$ large enough, the mean curvature is strictly positive if the starting ball stays above $\varphi$ (or the mean curvature is strictly negative if the starting ball stays below). This gives a contradiction with the Euler-Lagrange equation, i.e.  if the point is in $\Omega\times \R$ then $H_s[E](X_0)=0$, while if the point is on $\partial \Omega \times \R$ then $H_s[E](X_0)\leq 0$ if the starting ball stays above, and the opposite where the starting ball stays below $\varphi$. \\
		In the same way, we obtain a contradiction by using Remark \ref{geom_obst_locmin_remark}, taking into account that in addition, when we start with a ball "below",  the touching point might be on $\partial A\times \R$, in case $A \subset \subset \Omega$.
		
		This proves that there cannot be any touching points in $\overline \Omega\times \R$ in the strips $(-\infty, -t_0+R)$ and $(t_0-R, +\infty)$, which prove the claim \eqref{bdd}.
	\end{proof}
	
	\section{Proof of asymptotics of minimizers as $s\to0$} \label{two}
	
	%\subsection{State of the art and known results}\label{art}\quad \\

	\subsection{Preliminary results. Sets with positive  mean curvature}

	In this section we study the asymptotics as $s\to $ of a sequence of $s$-minimal functions
	$u_s\in\W_{\varphi, \psi}^s(\Omega,A)$.

	We underline that the set function $\bar\alpha$ introduced in \eqref{baralpha1} comes up also in
	the asymptotics of the fractional mean curvature as~$s\to0^+$, see \cite[Theorem 1.1]{BucurLombardini}. If $P\in\partial E$ and $\partial E$ is smooth around $P$, then 
	\bgs{ \liminf_{s\to0^+} s\,H_s[E](P) =\omega_{n+1} -2 \overline \alpha(E).}
	As a matter of fact, when $E$ is bounded, $\alpha(E)=0$ and the limit for $s\searrow 0$ of the fractional mean curvature is $\omega_{n+1},$ as already observed in
	Appendix~B in \cite{senonlocal}. This observation is interesting also in view of our stickiness results: if $\partial E$ is smooth around $Q$ and if $\overline \alpha(E_0) <\omega_{n+1}/2 $, then for $s$ small enough $H_s[E](Q)$ remains strictly positive and we obtain a contradiction with the Euler-Lagrange equation. We prove next that the mean curvature of $\partial E$ is actually strictly positive at all points that allow an exterior tangent ball of some radius $\delta>0$, for all $s$ under a certain threshold depending on $\delta$.

	\begin{theorem}\label{positivecurvature}
		Let $\Omega\subset\Rn$ be a bounded open set.  Let $ E_0 \subset \Co\Omega^\infty$ 
		%\varphi:\R^n\lra\R$ be such that
		%$$\varphi\in L^\infty_{loc}(\Co\Omega)\quad\mbox{ and }\quad
		be such that \[ \overline{\alpha}(E_0)
		<\frac{\omega_{n+1}}{2}\] and let $E\subset \R^{n+1}$ be such that \[E\setminus \Omega^\infty=E_0.\]
		We define 
		\eqlab{\beta=\beta(E_0):=\frac{\omega_{n+1}-2\overline \alpha(E_0)}4 \qquad \mbox{ and } \quad \qquad \label{delta_wild_index_def}
			\delta_s=\delta_s(E_0):=e^{-\frac{1}{s}\log \frac{\omega_{n+1}+2\beta}{\omega_{n+1}+\beta}} ,}
		for every $s\in(0,1)$.
		Then, for any $k \in \N\setminus\{0\} $ there exists $s_k=s_k(E_0,\Omega)\in(0,\frac{1}{2}]$ such that, if $E$ has an exterior tangent ball
		of radius (at least) $\delta_\sigma$, for some $\sigma\in(0,s_k]$, at some point $Q\in\partial E\cap{\left(\overline\Omega\times[-k,k]\right)}$, then
		\eqlab{\label{unif_pos_curv}\liminf_{\rho\to0^+}H_s^\rho[E](Q)\geq\frac{\beta}{s}>0,\qquad\forall\,s\in(0,\sigma].}
		\end{theorem}
	
	\begin{proof}The proof follows as in \cite[Theorem 1.2]{BucurLombardini} with some adaptation.  
		Let 
		\[ R=3 \max\{1,k,\mbox{diam}(\Omega)\} \] 
		and suppose that $E$ has an exterior tangent ball of radius $\delta<R/2$ at $Q\in\partial E\cap{\left(\overline\Omega\times[-k,k]\right)}$, that is
		\[\ball_\delta(P)\subset\Co E\quad\textrm{and}\quad Q\in\partial \ball_\delta(P)\]
		for some $P\in \R^{n+1}$. % using that $\overline \Omega^k \subset B_R(Q)$. 
		Then for $\rho<\delta/2$ small enough we have that
		\bgs{\label{uno} H^\rho_s[E](Q)
			= \int_{\ball_R(Q)\setminus \ball_{\rho}(Q)} \frac{\chi_{ \Co E} (Y)-\chi_{E}(Y)}{|Q-Y|^{n+1+s}}\, dY
			+ \int_{\Co \ball_R(Q)} \frac{\chi_{ \Co E} (Y)-\chi_{E}(Y)}{|Q-Y|^{n+1+s}}\, dY.}
		Then, as in Theorem 1.2 in \cite{BucurLombardini} we get that
		\bgs{\int_{\ball_R(Q)\setminus \ball_{\rho}(Q)} \frac{\chi_{ \Co E} (Y)-\chi_{E}(Y)}{|Q-Y|^{n+1+s}}\, dY 
			\geq
			-{\delta^{-s}} \left(C_1+ \frac{\omega_{n+1}}s \right) +\frac{\omega_{n+1}}s R^{-s}, }
		for some $C_1>0$ which does not depend on $s$. 
		On the other hand we obtain that
		\bgs{ & \int_{\Co \ball_R(Q)} \frac{\chi_{ \Co E} (Y)-\chi_{E}(Y)}{|Q-Y|^{n+1+s}}\, dY  
			\\ = &\; \int_{\Co \ball_R(Q)}\frac{dY}{|Q-Y|^{n+1+s}}- 2\int_{\Co \ball_R(Q)}\frac{\chi_E(Y)}{|Q-Y|^{n+1+s}}dY
			\\ 
			= &\;	\frac{\omega_{n+1}}s R^{-s} -2 \int_{\Co \ball_R(Q)\cap  \Omega^\infty }\frac{\chi_{E}(Y)}{|Q-Y|^{n+1+s}}\, dY - 2 \int_{\Co \ball_R(Q)\setminus \Omega^\infty}\frac{\chi_{E}(Y)}{|Q-Y|^{n+1+s}}\, dY 	 
			\\
			\geq &\; \frac{\omega_{n+1}}s R^{-s} -2 \int_{\Co \ball_R(Q)}\frac{\chi_{\Omega^\infty}(Y)}{|Q-Y|^{n+1+s}}\, dY- 2 \int_{\Co \ball_R(Q)}\frac{\chi_{E_0}(Y)}{|Q-Y|^{n+1+s}}\, dY  ,}
		using that $\chi_{E} \leq\chi_{\Omega^\infty} $ in $\Co \ball_R(Q) \cap \Omega^{\infty}$ and that $E_0 \cap \ball_R(Q) \supset E_0 \cap (\ball_R(Q) \setminus \Omega^\infty)$.  Denoting for any set $A\subset \R^{n+1}$
		\[ \alpha_s(Q,R,A):=\int_{\Co \ball_R(Q) }\frac{\chi_{A}(Y)}{|Q-Y|^{n+1+s}}\, dY  ,\]
		we obtain
		\bgs{ \int_{\Co \ball_R(Q)} \frac{\chi_{ \Co E} (Y)-\chi_{E}(Y)}{|Q-Y|^{n+1+s}}\, dY 
			\geq &\;\frac{\omega_{n+1}}s R^{-s} - 2\alpha_s(Q,R,\Omega^\infty)   -2 \alpha_s(Q, R, E_0)\\
			\geq &\;\frac{\omega_{n+1}}s R^{-s} - 2 \sup_{Q\in \overline{\Omega^k} }\alpha_s(Q,R,\Omega^\infty)-2\sup_{Q\in \overline{\Omega^k}  }  \alpha_s(Q, R, E_0) .
		}
		We recall from \cite[Proposition 2.3]{BucurLombardini} that for any two disjoint sets $E,F\subset \R^{n+1}$, such that $\alpha(E),\alpha(F),$ \\$\alpha(E\cup F)$ all exist we have that
		\[ \alpha(E\cup F)=\alpha(E)+\alpha(F),\] while for any $E,F\subset \R^{n+1}$ such that for some $r>0,Q\in \R^{n+1} $, $E\setminus \ball_r(Q)\subset F\setminus \ball_r(Q)$, then
		\[\overline \alpha(E)\leq \overline \alpha(F).\]
		Notice that $\Omega^\infty\subset A_{R} :=\{ X=(x_1, \dots,x_{n+1}) \in \R^{n+1} \;\big|\;  |x_1-Q_1| < R\} $, then
		\bgs{ \overline \alpha(\Omega^\infty) \leq \overline \alpha(A_R) = \alpha (\{ X \in \R^{n+1} \;\big|\; x_1 \geq Q_1-R\} )- \alpha (\{ X \in R^{n+1} \;\big|\; x_1 \geq Q_1+ R\})=0,}
		since the contribution from infinity of a half-space is $\omega_{n+1}/2.$ 
		One proceeds as in Theorem 1.2 in \cite{BucurLombardini} to get the conclusion. Precisely putting together the above estimates one gets for all $s\in (0,1)$
		\bgs{ H_s^\rho[E](Q) \geq &\;  
			\frac{1}s \left[-{\delta^{-s}}  \left(C_1 s   + \omega_{n+1}\right)+ {\omega_{n+1}} R^{-s}\right. 
			\\&\;  
			\left. 
			+ \left(\omega_{n+1} R^{-s} -2s \sup_{Q\in \overline{\Omega^k}  }  \alpha_s(Q, R, E_0)\right)\right. \\
			&\; \left. - 2 s \sup_{Q\in \overline{\Omega^k}} \alpha_s(Q, \R, \Omega^\infty) \right], 
		}
		and recalling that $ R>1, R^s\nearrow 1$ as $s\to 0$,  
		and from \cite[Proposition 2.1]{BucurLombardini} that
		\[ \limsup_{s \to 0} s \sup_{Q\in \overline{\Omega^k}} \alpha_s(Q,R,E_0) = \overline \alpha(E_0),\] one finds $\sigma\in (0,1)$ small enough such that for all $s\in (0, \sigma]$ 
		\bgs{ &\; C_1 s \leq \beta, \qquad \omega_{n+1} R^{-s} \geq \omega_{n+1}-\beta/4, \qquad 2 s \sup_{Q\in \overline{\Omega^k}} \alpha_s(Q, \R, \Omega^\infty) \leq \beta/4, \\ &\; \omega_{n+1} R^{-s} -2s \sup_{Q\in \overline{\Omega^k}  }  \alpha_s(Q, R, E_0) \geq 7\beta/2.} 
		For $\delta_\sigma<1$ given in the hypothesis one gets,
		\[ H_s^\rho[E](Q) \geq \frac{1}{s} \left[ -\delta_\sigma^{-s}(\beta + \omega_{n+1} ) + \omega_{n+1} +3\beta \right] \geq \frac{1}{s} \left[ -\delta_\sigma^{-\sigma}(\beta + \omega_{n+1} ) + \omega_{n+1} +3\beta \right] =\frac{\beta}s,\] 
		thus the conclusion of the Theorem.
	\end{proof}

	%%%%%%%SECTION%%%%%
	
	\subsection{Main results}
	
	In this section, we prove the main result of this section. In particular, we prove that (for a suitable exterior data), when $s$ is small enough the coincidence set of $u_s$ must be the whole of $A$.
	As a consequence, we show that the function $u_s$
	need not be continuous across $\partial\Omega$ or across $\partial A$.

	\smallskip

	\begin{theorem}\label{asympt_obst}
		Let $\Omega\subset\R^n$ be a bounded and connected open set with $C^2$ boundary and let $\varphi:\R^n\lra\R$ be such that
		$$\varphi\in L^\infty_{loc}(\Rn)\quad\mbox{ and }\quad\overline{\alpha}\big(\Sg(\varphi)\big)
		<\frac{\omega_{n+1}}{2}.$$
		Let $A\subset\subset\Omega$ be a bounded open set (eventually empty) with $C^2$ boundary. Let also
		\begin{itemize}
			\item [a)]  $\psi\in  C^2(\overline A)$. 
		\end{itemize}
		or
		\begin{itemize}
			\item [b)]  $\psi\in C(\overline A)\cap C^2(A)$ be such that
			$$\mathfrak O:= \Omega^\infty\setminus \{(x,t)\in\R^{n+1}\,|\,x\in\overline A,\,t\le\psi(x)\}$$
			is an open set with $C^2$ boundary. 
		\end{itemize}
		Define
		$$k_0:=2+\lceil\max\{\|\psi\|_{L^\infty(A)},\|\varphi\|_{L^\infty(\Omega_1\setminus\Omega)}\}\rceil,$$
		where $\lceil\cdot \rceil$ is the ceiling function.
		For every $s\in(0,1)$ let $u_s\in\W^s_{\varphi,\psi}(\Omega, A)$ be the unique $s$-minimal function.
		Then, for every $k\ge k_0$ there exists $s_k\in(0,1)$ such that
		\eqlab{
			u_s\le-k\quad\mbox{ a.e. in }\Omega\setminus A\quad\mbox{ and }
			\quad u_s=\psi\quad\mbox{ a.e. in }A,
		}
		for every $s\in(0,s_k)$.
		In particular
		$$\lim_{s\to0}u_s(x)=-\infty,\quad\mbox{ uniformly in }x\in\Omega\setminus A.$$
	\end{theorem}
	
	\begin{proof}
		We begin by proving point a) of the Theorem. Let 
		\[ \delta :=  \min\bigg\{r_0(A),r_0(\Omega),r_0(a), \frac14, \frac{d(\partial A, \partial \Omega)}{4}\bigg\},
		\]
		where $r_0(a)$ is given in Lemma \ref{palla}, with
		\[ a:= \frac{\|D^2\psi\|_{C^0(\overline A)}}4,\]
		and $r_0(A)$ and $r_0(\Omega)$ are as in Remark \ref{r0}.
		
		Let $\delta_s= \delta_s(E_0)$ and $s_k(E_0,\Omega)$ be as in Theorem \ref{positivecurvature}, with $E_0=\Sg(\varphi)\setminus \Omega^{\infty}$. Since $\delta_s\searrow 0$ as $s\to 0$, there exists $s_0\in (0, s_{k_0})$  small enough such that $\delta_{s_0} \leq \delta$. 
		\\
		
		\noindent \textbf{Step 1.} In the first step of the proof, we prove that $\Sg(u_s)$ is empty inside the strip $\overline \Omega \times (k_0-3\delta, +\infty)$. \\
		We begin by proving that $\Omega \times (k_0-3\delta, k_0-\delta) \subset\Sg(u_s)_{ext}$ for every $s<s_0$. The reader can see a sketch of this construction in Figure \ref{uno}.
		
		\noindent Suppose by contradiction that there exists a point laying in the subgraph  of $u_s$ and in the strip $\Omega \times (k_0-3\delta, k_0-\delta)$, more precisely
		\eqlab{
			\label{topo}
			\exists\; (x,y)\in \big(\Omega \times (k_0-3\delta, k_0-\delta)\big) \cap \overline{\Sg(u_s)}. 
		}
		We consider $\ball_{\delta}(q,y)$ for some $q\in \partial \Omega_{\delta}$. Notice that this ball is centered at a point with the coordinate $q$ that stays at distance $\delta$ from the boundary of $\Omega$, from the outside, and the same ``height'' as the considered point $(x,y)$. Recall also that $\Omega^\infty$ has a tangent ball of radius at least $r_0(\Omega)$ at all points on the ``walls'' of the infinite cylinder (given the smoothness of $\partial \Omega$). So, since $\delta\le r_0(\Omega)$, the ball $\ball_\delta(q,y)$
		is tangent from the outside to $\Omega^\infty$. Moreover, since $k_0\ge\|\varphi\|_{L^\infty(\Omega_1\setminus\Omega)}+2$, $y\ge k_0-3\delta$ and $\delta\le\frac14$,
		we have
		\eqlab{\label{castoro}
			\ball_\delta(q,y)\subset \Sg(\varphi)_{ext} \setminus\Omega^\infty=\Sg(u_s)_{ext}\setminus\Omega^\infty.
		}
		\begin{center}
			\begin{figure}[htpb]\label{uno}
				\hspace{0.9cm}
				\begin{minipage}[b]{0.9\linewidth}
					\centering
					\includegraphics[width=0.9\textwidth]{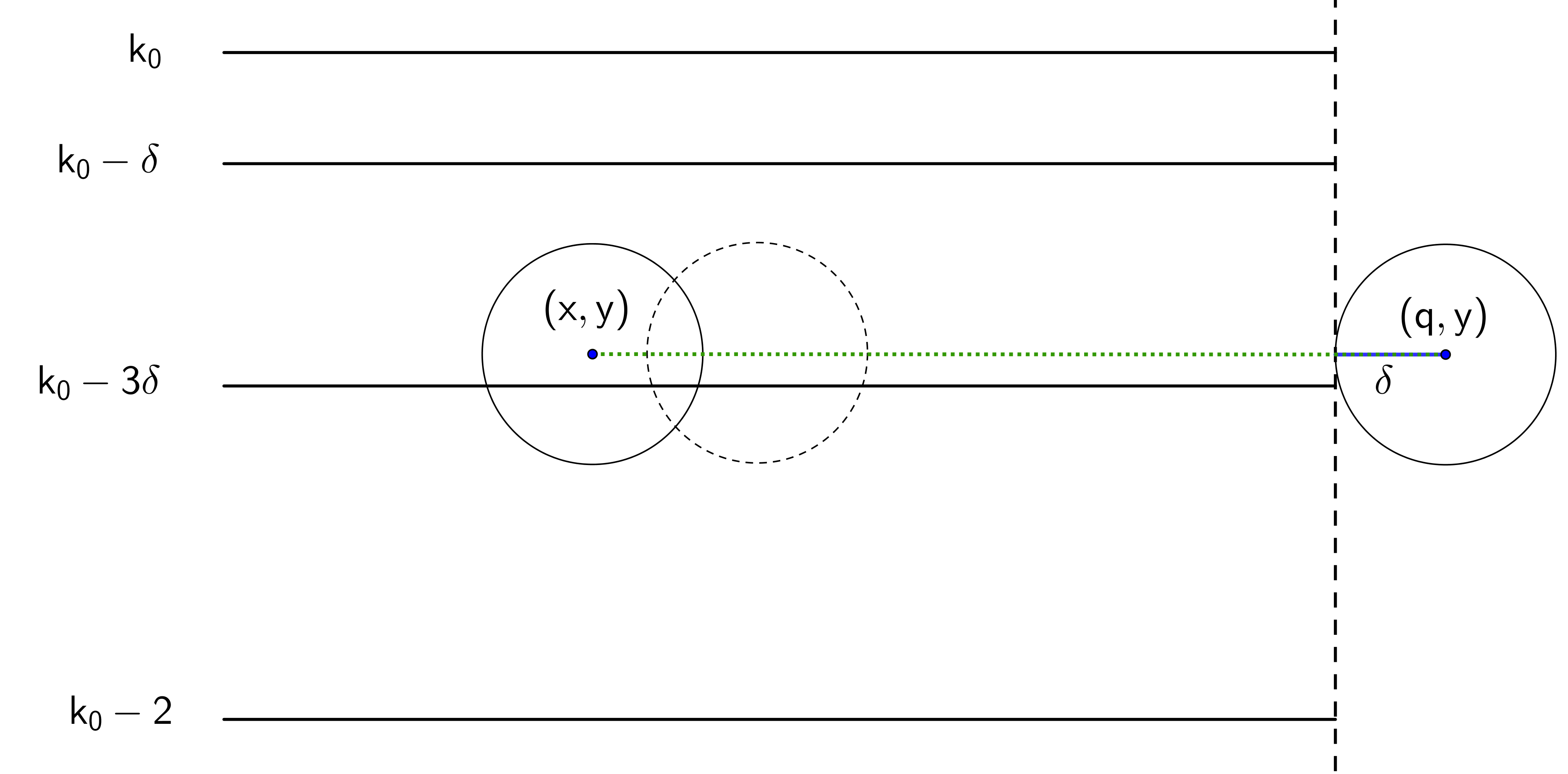}
					\caption{The construction to empty the strip $\Omega\times (k_0-3\delta, k_0-\delta)$}   
					%	\label{uno456}
					\label{uno}
				\end{minipage}
			\end{figure} 
		\end{center}
		
		We consider the segment
		\[
		\gamma:[0,1]\lra\R^{n+1},\qquad\gamma(t):=(tq+(1-t)x,y),
		\]
		which connects the point $(q,y)$ to $(x,y)$. 
		We ``slide the ball'' along the segment $\gamma(t)$, as follows. Let 
		\[
		t_0:= \sup\Big\{ \tau\in [0,1] \; \big | \; \bigcup_{t\in [0,\tau]} \ball_{\delta} (\gamma(t)) \subset \Sg(u_s)_{ext}\Big\}.
		\]
		Notice that, by \eqref{castoro}, the first ball is in the complement of the subgraph,
		\[
		\ball_\delta(\gamma(0))\subset\Sg(u_s)_{ext}
		\]
		and by assumption the last ball has to touch the subgraph, i.e.
		\[
		\ball_\delta(\gamma(1))\cap\overline{\Sg(u_s)}\not=\emptyset,
		\]
		hence $t_0\in[0,1)$. Arguing as in Lemma A.1 in \cite{BucurLombardini}, we get a tangency point between the ball centered at $\gamma(t_0)$ and the boundary of the subgraph of $u_s$, that is
		\[ 
		\ball_{\delta} (\gamma(t_0)) \subset \Sg(u_s)_{ext} \qquad \mbox{ and } \qquad \exists (p,z)\in \partial \ball_\delta(\gamma(t_0)) \cap \partial \Sg(u_s).
		\]  
		By definition of $\gamma$, we have that either $(p,z)\in \Omega^\infty$ or $(p,z)\in \partial \Omega^\infty$. 
		In the first case, by Theorem 5.1 in \cite{nms} we have that
		\eqlab{\label{criceto}
			H_s[\Sg(u_s)](p,z) \leq 0.
		}
		In the second case,  we observe that 
		\[
		z\geq k_0-4\delta\geq k_0-1 \qquad \mbox{ hence }  \qquad \ball_{\frac12}(p,z) \setminus \Omega^\infty \subset  \Sg(u_s)_{ext}.
		\]
		The hypothesis of Theorem 1.1 in \cite{elsulbordo} hold, so we still have the inequality \eqref{criceto}. On  the other hand, given that $z\in [k_0-1,k_0]$, $s_0\leq s_{k_0}$ and $\delta \geq \delta_{s_{0}}$, by Theorem \ref{positivecurvature} we have that
		\[ 
		H_s[\Sg(u_s)](p,z) >0,
		\]
		which gives a contradiction with \eqref{criceto}. 
		This proves that \eqref{topo} cannot hold. So 
		\[
		\Omega\times (k_0-3\delta, k_0-\delta) \subset \Sg(u_s)_{ext},
		\]
		We remark that since $(\Omega_1\setminus \Omega) \times (k_0-1,\infty) \subset  \Sg(u_s)_{ext}$, we actually have 
		\[
		\overline  \Omega\times [k_0-3\delta, k_0-\delta] \subset  \Sg(u_s)_{ext}.
		\]
		Furthermore, given that $\Sg(u_s)$ is a subgraph, we have that
		\[
		\overline \Omega \times [k_0-3\delta,+\infty) \subset  \Sg(u_s)_{ext}.
		\]
		
		\noindent \textbf{Step 2.} In the second step of the proof, we prove that $\Sg(u_s)$ is empty inside $\left(\Omega\setminus \overline A \right) \times[-k,k_0]$.\\ We take $k\geq k_0$ and we claim that for every $s\in (0,\min\{s_0,s_{k+1}\})$, where $s_{k+1}$ is given in Theorem \ref{positivecurvature}, we have that
		\eqlab{\label{rock} (\Omega_{-\eps} \setminus \overline {A_\eps})\times (-k,k_0) \subset  \Sg(u_s)_{ext}
		}
		for every $\eps \in (0, \frac{\delta}8)$. 
		Notice that, by the definition of $\delta_s$, we have that 
		\[ 
		\delta_s\leq \delta_{s_0}\leq\delta.
		\]
		Thus, by Theorem \ref{positivecurvature} we have that 	if $\Sg(u_s)$ has an exterior tangent ball
		of radius (at least) $\delta$ at some point $X_0 \in\partial \Sg(u_s) \cap{\left(\overline\Omega\times[-k-1,k_0+1]\right)}$, then
		\eqlab{\label{rocky} 
			\liminf_{\rho\to0^+}H_s^\rho[\Sg(u_s)](X_0)\geq\frac{\beta}{s}>0.
		}
		Now we prove \eqref{rock}. Suppose by contradiction that 
		\[ 
		\exists \; \eps \in \left(0, \frac{\delta}8\right) \mbox{ and } \exists \; (p,z)\in \Big( (\Omega_{-\eps} \setminus \overline {A_\eps})\times (-k,k_0 )\Big)\cap \overline {\Sg(u_s)}. 
		\] 
		\begin{center}
			\begin{figure}[htpb]
				\hspace{0.80cm}
				\begin{minipage}[b]{0.80\linewidth}
					\centering
					\includegraphics[width=0.80\textwidth]{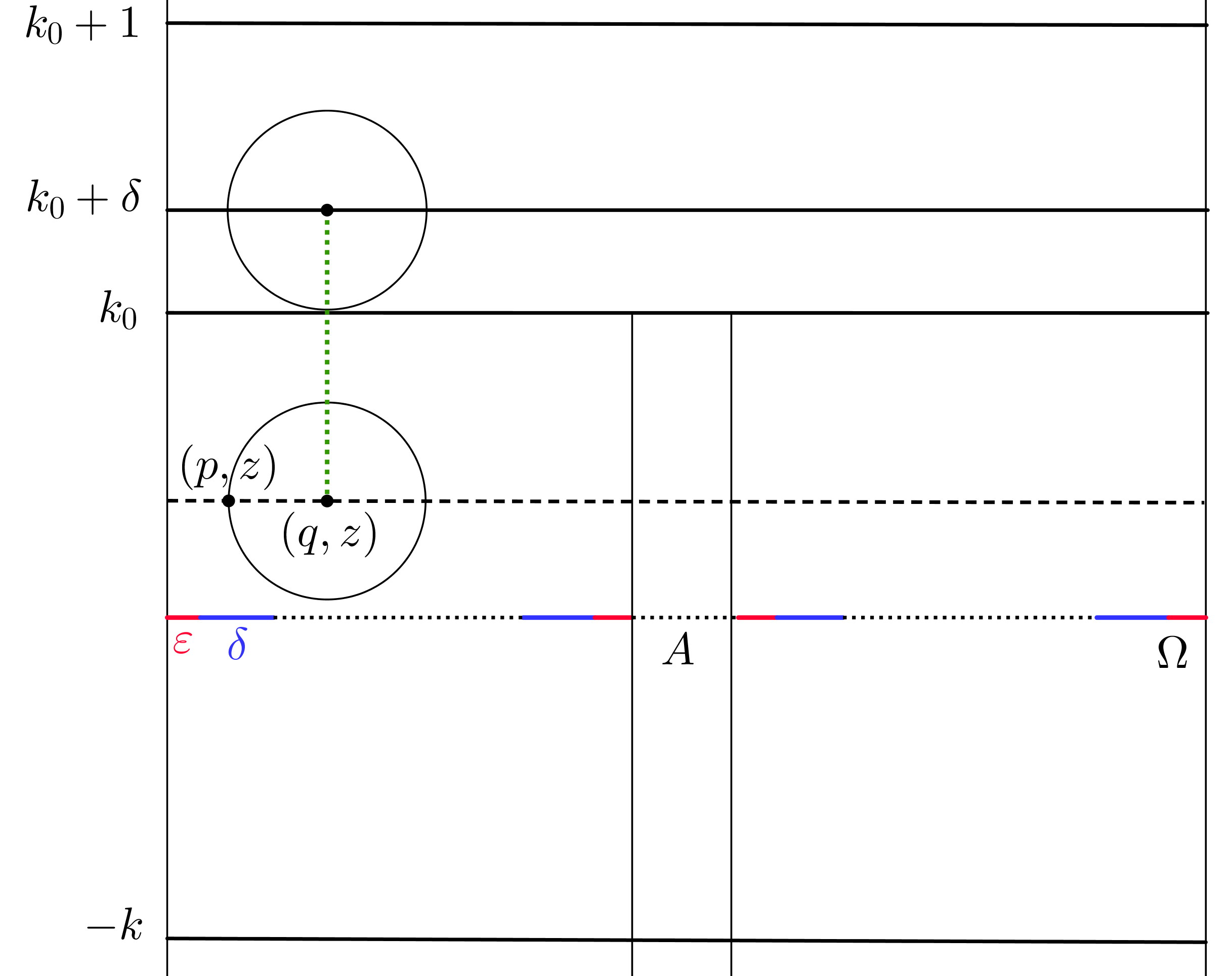}
					\caption{The construction to empty $\left(\Omega\setminus \overline A\right)  \times (-k, k_0)$}   
					\label{due}
				\end{minipage}
			\end{figure} 
		\end{center}
		In this case, we remark that $(p,z)\in \ball_\delta(q,z)$ for some $q\in \Omega_{-\eps-\delta } \setminus \overline {A_{\eps+\delta}}$ (see Figure \ref{due}). With this choice, we have that 
		\[
		B_{\delta}(q,t) \subset \subset \left(\Omega \setminus \overline A\right) \times (-k-1,k_0+1),\]
		for every $t\in [z,k_0+\delta]$.  Furthermore from the previous step of the proof we know that
		\[ 
		\ball_\delta(q,k_0+\delta) \subset \Sg(u_s)_{ext}.\]
		We take the segment connecting $(q,k_0+\delta)$ to $(q,z)$ and argue as in the first step of the proof, getting a contradiction between \eqref{rocky} and the Euler-Lagrange equation (Theorem 5.1 in \cite{nms}). This means that 
		\[
		\left( \Omega\setminus \overline A \right) \times[-k,k_0]\subset  \Sg(u_s)_{ext}.
		\]
		\begin{center}
			\begin{figure}[htpb]
				\hspace{0.77cm}
				\begin{minipage}[b]{0.77\linewidth}
					\centering
					\includegraphics[width=0.77\textwidth]{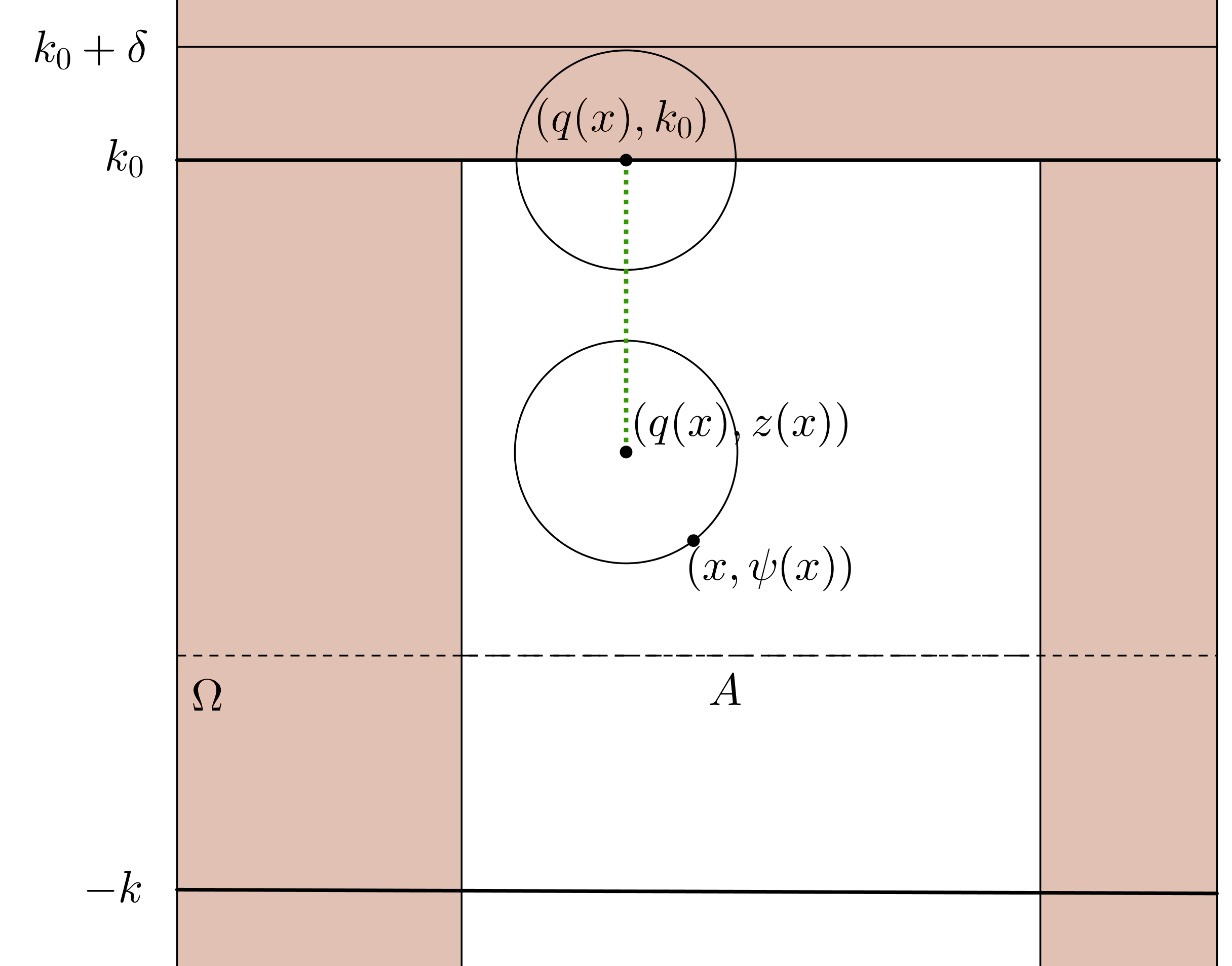}
					\caption{The construction to empty $\{ (x,y) \in \R^{n+1} \; \big| \; x\in\overline A,\, \psi(x)<y<k_0\}$}   
					\label{tre}
				\end{minipage}
			\end{figure} 
		\end{center}
		\noindent \textbf{Step 3.} 
		We are going to show in this third part of the proof that
		\eqlab{\label{sss} \{ (x,y) \in \R^{n+1} \; \big| \; x\in\overline A,\, \psi(x)<y<k_0\}\subset  \Sg(u_s)_{ext} 
		}(see Figure \ref{tre}).
		Notice that, since
		\[ \delta\leq r_0(a), \qquad \mbox{ with } \quad a= \frac{\|D^2 \psi\|_{C^0(\overline A)}}4
		\]
		by  Lemma \ref{parab} and  Lemma \ref{palla}, for every $x\in \overline A$, we can find a ball tangent to the graph of $\psi$ at $(x,\psi(x))$, i.e.
		\[B_\delta(q(x),z(x)) \subset \left(\Omega \times [-k,k_0]\right)  \setminus  \{ (x,y) \in \R^{n+1} \; \big| \; x\in \overline A, y \leq \psi(x) \} 
		\] 
		such that
		\[ (x, \psi(x))\in \partial B_\delta(q(x),z(x)).\]
		Then it is clear that
		\[B_\delta(q(x),y) \subset \left(\Omega \times [-k,k_0+\delta]\right)  \setminus  \{ (x,y) \in \R^{n+1} \; \big| \; x\in \overline A, y \leq \psi(x) \}  \qquad \mbox{ for every }  y\in[z(x), k_0].
		\]
		We prove \eqref{sss} by contradiction, supposing that
		\[ \exists \; (p,h)\in \R^{n+1}, \quad p\in \overline A, \; \psi(p)<h<k_0 \quad \mbox{ such that } (p,h)\in  \overline {\Sg(u_s)}.
		\]
		We consider the segment that connects $(q(p),k_0)$ to $(q(p),z(p))$ and ``slide the ball'' $B_{\delta}(q(p),k_0)$ as in the first part of the proof, until we find the first touching point on $\partial \Sg(u_s)$. Any such point provides a contradiction between the Euler-Lagrange theorem, and the positivity of the mean curvature (as given by Theorem \ref{positivecurvature}). This gives a proof of \eqref{sss} and concludes point a) of the Theorem.\\
		Point b) of the Theorem is proved in the same way, by taking the radius of the ball  smaller than $r_0(\mathfrak O)$, instead of $r_0(a)$. More precisely we consider
		\[ \delta :=  \min\bigg\{r_0(A),r_0(\Omega),r_0(\mathfrak O), \frac14, \frac{d(\partial A, \partial \Omega)}{4}\bigg\}.
		\]
		This concludes the proof of the Theorem.
	\end{proof}
	%\textcolor{blue}{Aggiungere distanza segnata e $r_0$ degli insieme $C2$.}
	
	A similar result holds for an everywhere defined obstacle, that is for the case $A=\Omega$, provided
	that the set
	$$\{(x,t)\in\Omega^\infty\,|\,t>\psi(x)\}$$
	is an open set with $C^2$ boundary.
	Then, there exists $s_0\in(0,1)$ such that
	$$u_s=\psi\quad\mbox{ a.e. in }\Omega,$$
	for every $s<s_0$, meaning that below some threshold $s_0$, the minimizers coincide with the obstacle.
	
	\begin{theorem}\label{all}
		Let $\Omega\subset\R^n$ be a bounded and connected open set with $C^2$ boundary and let $\varphi:\R^n\lra\R$ be such that
		$$\varphi\in L^\infty_{loc}(\Rn)\quad\mbox{ and }\quad\overline{\alpha}\big(\Sg(\varphi)\big)
		<\frac{\omega_{n+1}}{2}.$$
		Let also
		\begin{itemize}
			\item [a)] $\psi\in  C^2(\overline \Omega)$ and there exist $\varrho,a>0$
			such that $\varphi \in C^0(\Omega_{2\varrho}\setminus \Omega)$ and for every $x_0\in\partial\Omega$ it holds
			\eqlab{\label{paraboloid_cond}
				\varphi(x)\le\psi(x_0)+\nabla\psi(x_0)\cdot(x-x_0)+\frac{a}{2}|x-x_0|^2
				\quad \forall \;  x\in B_\varrho(x_0)\setminus\Omega,}
		\end{itemize}
		or
		\begin{itemize}
			\item [b)]   $\psi\in C(\overline \Omega)\cap C^2(\Omega)$ be such that
			$$\mathfrak O:=\{(x,t)\in\Omega^\infty\,|\,t>\psi(x)\}$$
			is an open set with $C^2$ boundary,
		\end{itemize}
		For every $s\in(0,1)$ let $u_s\in\W^s_{\varphi,\psi}(\Omega, \Omega)$ be the unique $s$-minimal function.
		Then there exists $s_0\in(0,1)$ such that
		\eqlab{
			u_s=\psi\quad\mbox{ a.e. in }\Omega,
		}
		for every $s\in(0,s_0)$.
	\end{theorem}
	
	\begin{proof}In order the prove point a) of this theorem, we define 
		\[ 
		k_0:=2+\max\{ \|\varphi\|_{L^\infty(\Omega_1\setminus \Omega)} , \|\psi\|_{L^\infty(\Omega)}\}	
		\] 
		and
		\[ b=\max\left\{ a, \frac{\|D^2\psi\|_{C^0(\overline \Omega)}}2\right\} \qquad 
		\delta:=\min\bigg\{r_0(\Omega),r_0(b), \frac14\bigg\}.
		\]
		
		We proceed as in the \emph{Step 1} of Theorem \ref{asympt_obst} and ``empty'' the strip $\overline\Omega \times [k_0-3\delta,+\infty).$  After that, by using Lemma \ref{parab} and \ref{palla}, we go on as in \emph{Step 3} of the above mentioned Theorem \ref{asympt_obst} and prove that $\Sg(u_s)$ is empty inside $\mathfrak O$. 
		For point b), we take 
		\[ 
		\delta:=\min\bigg\{r_0(\mathfrak O), \frac14\bigg\}.
		\]
		We then ``reflect'' the obstacle, in order to obtain a $C^2$ set. We define for any $k\geq k_0$ the set 
		\[ \Op(k) := \{ (x,y)\in \R^{n+1} \, \big| \, \psi(x)<y<k+\|\psi\|_{L^\infty(\Omega)}-\psi(x)\}.\] Notice that $\Op(k)$ is bounded and has $C^2$ boundary, thanks to the hypothesis. Furthermore, since $k\geq k_0$, by definition of $k_0$ we have that the exterior data $\Sg(\varphi)$ does not surround $\Op(k)$. We can argue as in the proof of  Theorem 1.7 in \cite{BucurLombardini}, by using Theorem \ref{positivecurvature} instead of \cite[ Theorem 1.2]{BucurLombardini}.  
	\end{proof}

	\begin{remark}\label{rmkwitheps}
		Notice that if instead of $\psi$ we consider the rescaled function $\eps \psi$, with $\eps\in [0,1]$ fixed,  the stickiness still happens. Thus, the stickiness is not conditioned by how large the obstacle is, in contrast to the classical case. Moreover, the threshold fractional parameter $s_0>0$ does not depend on $\eps$, but only on the exterior data (in particular of the function $\alpha$ of the exterior data) and on the "original" obstacle function $\psi$. 
	\end{remark}
	On the other hand, when
	$$\underline{\alpha}\big(\Sg(\varphi)\big)>\frac{\omega_{n+1}}{2},$$
	the $s$-minimal functions $u_s$ completely detach from the obstacle $\psi$, meaning that the coincidence set is empty. 
	
	\begin{theorem}\label{opposite}
		Let $\Omega\subset \Rn$ be a bounded and connected open set with $C^2$ boundary. Let $\varphi\colon\Rn  \lra \R$ be such that
		$$\varphi\in L^\infty_{loc}(\Co\Omega)\quad\mbox{ and }\quad\underline{\alpha}\big(\Sg(\varphi)\big)
		>\frac{\omega_{n+1}}{2}.$$
		Let $A \subset \subset\Omega$ be open (eventually $A=\emptyset$ or $A=\Omega$). Let also
		$\psi\in L^\infty(A)$ (and $\psi\in C^2_{loc}(\Omega)$ when $A=\Omega)$). 
		Define
		$$k_0:=2+\max\{\|\psi\|_{L^\infty(A)},\|\varphi\|_{L^\infty(\Omega_2\setminus\Omega)}\}.$$
		For every $s\in(0,1)$ let $u_s\in\W_{\varphi,\psi}^s(\Omega,A)$ be the unique $s$-minimal function.
		Then, for every $k\ge k_0$ there exists $s_k\in(0,1)$ such that
		\eqlab{
			u_s\ge k\quad\mbox{ a.e. in }\Omega,
		}
		for every $s\in(0,s_k)$.
		In particular
		$$\lim_{s\to0}u_s(x)=\infty,\quad\mbox{ uniformly in }x\in\Omega.$$
	\end{theorem}
	\begin{proof}[Sketch of the proof]
		The proof is just an adaptation of the argument in \emph{Step 1} of Theorem \ref{asympt_obst}. We first notice that the counterpart of Theorem \ref{positivecurvature} for 
		$$\underline{\alpha}\big(\Sg(\varphi)\big)>\frac{\omega_{n+1}}{2},$$
		affirms that at any point on the boundary of the $\Sg(u_s)$, the fractional mean curvature is negative. Then, the idea is to start with a ``full'' ball. If $A\subset \subset \Omega$ (or $A=\emptyset$), then let 
		\[ 
		\delta:=\min\left\{ r_0(\Omega), \frac{ d(\partial A, \partial \Omega)}4 ,\frac{1}4\right\}.
		\]  
		Then take $\ball_\delta(x,y)$, with $x\in \partial \Omega_\delta$ and $y<\varphi(x)-2$.  We slide the ball horizontally, until we get into the position $\ball_\delta(z,y)$ with $z\in \partial \Omega_{-\delta}$ and then vertically up until $\ball_\delta(z,k_0-2\delta)$ and then again horizontally and vertically until we ``fill'' the strip $\Omega \times [k_0-3\delta, k_0-\delta]$. 
		
		On the other hand, if $A=\Omega$ and $\psi \in C^2_{loc}(\Omega)$, let then let
		\[ \delta:= \min\left\{ r_0(\Omega), \frac14\right\}.\] Then we start with $\ball_\delta(x,y)$, with $y\leq \psi(x)-2$, we slide the ball vertically until we reach $\ball_\delta(x,k_0-3\delta)$ and then horizontally and vertically until we ``fill'' the strip  $\Omega \times [k_0-3\delta, k_0-\delta]$.
		Given that $\Sg(u_s)$ is s subgraph, it means that $u_s\geq k$ for all $k\geq k_0$. 
	\end{proof}

	%%%%%%%%SECTION%%%%%%%%
	\appendix
	
	\section{Some geometric observations}\label{appendicite}
	
	\subsection{Tangent paraboloids}
	
	We make in this subsection some elementary remarks on tangent paraboloids,  which we insert for completeness.
	\begin{lemma}\label{parab}
		Let $\Op\subset\R^n$ be a bounded open set with $C^2$ boundary and let $\psi\in C^2(\overline{\Op})$.
		Then there exists $\varrho>0$ such that for every $x_0\in\overline{\Op}$ it holds
		\bgs{
			\psi(x)\le\psi(x_0)+\nabla\psi(x_0)\cdot(x-x_0)+\frac{\|D^2\psi\|_{C^0(\overline{\Op})}}{{2}}|x-x_0|^2,
			\quad\forall\,x\in\overline{\Op}\cap B_\varrho(x_0).
		}
	\end{lemma}
	This Lemma follows easily by using the Taylor expansion for the function $\psi$.  Notice that the vertex of the paraboloid is not generally at $(x_0,\psi(x_0))$ (unless the first order term vanishes), so the Lemma asserts that the function $\psi$ lies locally beneath part of the "branch'' of the paraboloid. 
	\begin{lemma}\label{palla}
		Let $a,c \in \R$ with $a>0$ and $b,x_0\in \Rn$ and define 
		\[ P(x):=c+ b\cdot (x-x_0) + \frac{a}2 |x-x_0|^2.\]
		Then there exists $r_0=r_0(a)>0$ such that $ \Sg(P)$ has an exterior tangent ball of radius $r_0$ at every point $(x,P(x))\in \partial \Sg(P)$.    
	\end{lemma}
	
	\begin{proof}
		Essentially what we want to do is to consider the osculating ball of the paraboloid at the vertex of the paraboloid. Sliding this ball along the "branches'', we still end up with an exterior tangent ball (of uniform radius) at every point on the boundary of the paraboloid.
		
		Let us consider the paraboloid $P(x)$, that we rewrite as
		\[ P(x)= \frac{A}2 |x|^2 + B\cdot x+ C, \quad \mbox{ where } A=a, \, B= b-ax_0,\, C=\frac{a}2|x_o|^2 -b\cdot x_0 +c,\]
		of vertex 
		\[ x_v= -\frac{B}{A} , \,\, y_v:=P(x_v)= -\frac{1}{2A} |B|^2 +C \]
		and the sphere $\partial \mathcal B_{\frac{1}{A}} (x_v, y_v+\frac{1}{A}) \subset \R^{n+1}$, which we claim to be the osculating ball. 
		We prove that the " inferior half-sphere'', i.e.
		\[ y=y_v+\frac{1}{A} -\sqrt{\frac{1}{A^2} -|x-x_v|^2}, \quad \mbox{ with } |x-x_v|\leq \frac{1}{A} \]
		lays above the paraboloid $P(x)$.
		Some algebraic manipulation show that for any $x\in B_{\frac{1}{A}}(x_v)\subset \Rn$ the inequality
		\[ y_v+\frac{1}{A} -\sqrt{\frac{1}{A^2} -|x-x_v|^2 }\geq \frac{A}2 |x|^2 + B\cdot x+ C\]
		is equivalent to having
		\[  2-|Ax+B|^2 \geq \sqrt {4-4|Ax+B|^2},\]
		which trivially holds. This shows the "inferior half-ball'', thus the whole ball of radius ${1}/{A}$ tangent to the paraboloid at the vertex of the paraboloid, lays inside the paraboloid itself (hence it is an exterior tangent ball to the subgraph of $P$). 
		This proves the Lemma. 
	\end{proof}

	\subsection{Smooth domes over cylinders}
	In this section, we indicate an alternative proof for the results in this paper. As a matter of fact,  we know from the proof of \cite[Theorem 1.4]{BucurLombardini} that if $\Omega$ has $C^2$ boundary, one can get a ball of uniform radius and can slide it inside the domain (much as what we did in the proof of the main result), in any direction. In our proof of Theorem \ref{asympt_obst}, we had to enter the domain horizontally and then move vertically, exactly because of the lack of smooth boundary (indeed, we worked in the presence of the corners of $\Omega^k$). So, if one could obtain a $C^2$ approximation of the cylinder $\Omega^k$, then the technicality of the proof of Theorem \ref{asympt_obst} can be overcome, and one could apply directly \cite[Theorem 1.4]{BucurLombardini} to get the main result in Theorem \ref{asympt_obst}.
	
	So, in this Appendix we will give a constructive approach to obtain a $C^2$ domain from a cylinder.
	\begin{remark}\label{r0}
		Recall that a bounded open set $\Op\subset\R^n$ has $C^k$ boundary, for some $k\ge2$
		if and only if the signed distance function $\bar{d}_\Op$ is $C^k$ in a tubular neighborhood of $\partial\Op$,
		that is,
		\eqlab{
			\bar{d}_\Op\in C^k\big(N_{3r_0}(\partial\Op)\big),
		}
		for some $r_0>0$. We denote such $r_0$ by $r_0(\Op)$. 
	\end{remark}

	\begin{prop}
		Let $\Op\subset\R^n$ be a bounded open set with $C^k$ boundary, for some $k\ge2$.
		Let
		\eqlab{
			\mathfrak b_\Op(x):=
			\big(-\bar{d}_\Op(x)\big)^\frac{1}{k+1},\quad\forall\, x\in\overline{\Op}\setminus\Op_{-3r_0},
		}
		and let $\eta\in C^\infty_c(\R^n)$ be a cut-off function such that
		\bgs{
			0\le\eta\le1,\quad\eta\equiv1\mbox{ in }\Op_{-2r_0}\quad\mbox{ and }
			\quad\mbox{supp }\eta\subset\Op_{-r_0}.
		}
		Given a function $u\in C^k(\Op)$, let
		\[\mathfrak t_u^+(x):=\eta(x)u(x)+\big(1-\eta(x)\big)\mathfrak b_\Op(x),\quad\forall\,x\in\overline{\Op},\]
		and
		define the ``upward dome'' of base $\Op$ and top $u$ as
		\eqlab{
			\mathcal D^+(\Op,u):=\left\{(x,t)\in\R^{n+1}\,|\,x\in\Op\mbox{ and }-\big(\|u\|_{L^\infty(\Op_{-r_0})}+1\big)
			< t<\mathfrak t_u^+(x)\right\},
		}
		Then
		\[\mathfrak t^+_u\in C^0(\overline{\Op})\cap
		C^k(\Op),\]
		and
		\[\partial\mathcal D^+(\Op,u)\mbox{ is of class $C^k$ in }
		\left\{x_{n+1}>-\big(\|u\|_{L^\infty(\Op_{-r_0})}+1\big)\right\}.\]
	\end{prop}
	
	\begin{proof}
		First of all,
		notice that
		\[\mathfrak t^+_u(x)=0\qquad\forall\,x\in\partial\Op\]
		and
		\[\partial\mathcal D^+(\Op,u)\cap
		\left\{x_{n+1}>-\|u\|_{L^\infty(\Op_{-r_0})}-1\right\}=D_1\cup D_2\cup\big(\partial\Op\times\{0\}\big),\]
		where
		\[D_1:=\left\{\big(x,\mathfrak t^+_u(x)\big)\in\R^{n+1}\,|\,x\in\Op\right\}\quad\mbox{ and }\quad
		D_2:=\partial\Op\times\big(-\|u\|_{L^\infty(\Op_{-r_0})}-1,0\big).\]
		We remark that both $D_1$ and $D_2$ are of class $C^k$.
		Hence, we only need to prove that $\partial\mathcal D^+(\Op,u)$ is of class $C^k$ in a neighborhood of
		$\partial\Op\times\{0\}$, that is where the two pieces $D_1$ and $D_2$ connect.
		
		In order to do this, let
		\[\mathcal V:=N_\delta(\partial\Op)\times(-1,1)\subset\R^{n+1},\]
		for some fixed $\delta>0$ small enough, say $\delta:=\min\{r_0,1\}/2$, and define the function
		\begin{equation*}
		F:\mathcal V\lra\R,\qquad F(x,t):=
		\left\{\begin{array}{cc}
		\bar{d}_\Op(x) & \mbox{if }t\le0,\\
		\bar{d}_\Op(x)+t^{k+1} & \mbox{if }t>0
		\end{array}\right.
		=\bar{d}_\Op(x)+(t_+)^{k+1},
		\end{equation*}
		where
		\[t_+:=\max\{0,t\}.\]
		Notice that the function
		\[t\in\R\longmapsto (t_+)^{k+1}\]
		is in $C^k(\R)$. Therefore $F\in C^k(\mathcal V)$.
		
		Moreover, recall that $|\nabla\bar{d}_\Op|=1$ whenever $\bar{d}_\Op$ is differentiable.
		Thus, since
		\bgs{
			\nabla F(x,t)=\left(\nabla\bar{d}_\Op(x),(k+1)(t_+)^k\right),
		}
		we have
		\eqlab{\label{Dini_ok}
			\nabla F(x,t)\not=0,\qquad\forall\,(x,t)\in \mathcal V.
		}
		
		Then we claim that
		\eqlab{\label{loc_def1}
			\mathcal D^+(\Op,u)\cap\mathcal V=\left\{(x,t)\in\mathcal V\,|\,F(x,t)<0\right\}
		}
		and
		\eqlab{\label{loc_def2}
			\partial\mathcal D^+(\Op,u)\cap\mathcal V=\left\{(x,t)\in\mathcal V\,|\,F(x,t)=0\right\}.
		}
		Indeed, since $\delta$ is small, we have
		\bgs{
			\mathcal D^+(\Op,u)\cap\mathcal V&=\left\{(x,t)\in\R^{n+1}\,|\,x\in N_\delta(\partial\Op)\cap\Op
			\mbox{ and }-1<t<\mathfrak b_\Op(x)\right\}\\
			&
			=\left\{(x,t)\in\R^{n+1}\,|\,-\delta<\bar{d}_\Op(x)<0
			\mbox{ and }-1<t<\big(-\bar{d}_\Op(x)\big)^\frac{1}{k+1}\right\}\\
			&
			=\left\{(x,t)\in\mathcal V\,|\,\bar{d}_\Op(x)<0
			\mbox{ and }t<\big(-\bar{d}_\Op(x)\big)^\frac{1}{k+1}\right\}.
		}
		Now let $(x,t)\in\mathcal V$. If $t\le0$, then
		\bgs{
			(x,t)\in\mathcal D^+(\Op,u)\quad\Longleftrightarrow\quad
			\bar{d}_\Op(x)<0\quad\Longleftrightarrow\quad
			F(x,t)<0.
		}
		On the other hand, if $t>0$, then
		\bgs{
			(x,t)\in\mathcal D^+(\Op,u)\quad\Longrightarrow\quad
			t^{k+1}<-\bar{d}_\Op(x)
			\quad\Longrightarrow\quad
			F(x,t)<0,
		}
		and
		\bgs{
			F(x,t)<0
			\quad\Longrightarrow\quad
			\bar{d}_\Op(x)<-t^{k+1}<0
			\quad\Longrightarrow\quad
			(x,t)\in\mathcal D^+(\Op,u).
		}
		This proves \eqref{loc_def1}
		and \eqref{loc_def2}
		is proved similarly.
		\\		
		Therefore, thanks to \eqref{Dini_ok}, we can conclude the proof of the Proposition
		by using Dini's Implicit Function Theorem.
	\end{proof}

	\begin{remark}
		We point out that the corresponding result holds true for the ``downward dome''
		\bgs{
			\mathcal D^-(\mathcal O,u)&:=\left\{(x,t)\in\R^{n+1}\,|\,x\in\mathcal O \mbox{ and }\mathfrak t_u^-(x)
			<t<\|u\|_{L^\infty(\mathcal O_{-r_0})}+1
			\right\}\\
			&
			=-\mathcal D^+(\mathcal O,-u),
		}
		where
		\[\mathfrak t_u^-(x):=\eta(x)u(x)-\big(1-\eta(x)\big)\mathfrak b_\Op(x)
		=-\mathfrak t^+_{-u}(x),\quad\forall\,x\in\overline{\Op}.\]
	\end{remark}

	\bibliography{biblio}
	\bibliographystyle{plain}

\end{document}